\documentclass[12pt]{amsart}
\usepackage[mathscr]{eucal}
\usepackage[margin=1.25in]{geometry}
\usepackage{amsfonts,amsmath,amssymb,amscd}
\usepackage{xcolor}
\usepackage{datetime}
\usepackage{comment}
\usepackage[colorlinks,citecolor=blue]{hyperref}

\newtheorem{theorem}{Theorem}[section]
\newtheorem{proposition}[theorem]{Proposition}
\newtheorem{lemma}[theorem]{Lemma}
\newtheorem{corollary}[theorem]{Corollary}
\theoremstyle{definition}

\newtheorem{example}[theorem]{Example}

\newtheorem{remark}[theorem]{Remark}

\def \D{\mathbb{D}}
\def \C{\mathbb{C}}
\def \T{\mathbb{T}}

\def \N{\mathbb{N}}
\newcommand{\clb}{\mathcal{B}}

\newcommand{\clk}{\mathcal{K}}

\newcommand{\clf}{\mathcal{F}}

\newcommand{\na}{\mathcal{N}}

\newcommand{\an}{\mathcal{AN}}
\newcommand{\am}{\mathcal{AM}}
\newcommand{\es}{\sigma_{\mathrm{ess}}}
\newcommand{\vp}{\varphi}

\newcommand{\Fog}{\mathscr{F}}

\numberwithin{equation}{section}



\author{Puspendu Nag}
\address{Department of Mathematics, IIT Hyderabad, 
Kandi, Sangareddy, Telangana, India, 502\,284.}
\email{ma23resch11003@iith.ac.in; puspomath@gmail.com}

\author{Ramesh Golla}
\address{Department of Mathematics, IIT Hyderabad, 
Kandi, Sangareddy, Telangana, India, 502\,284.}
\email{rameshg@math.iith.ac.in}

\title[Norm and Minimum Attainment of Foguel-Type Operators]{Norm and Minimum Attainment of Foguel-Type Operators}

\begin{document}

\begin{abstract}
In this article, we study Foguel-type operators
$\EuScript{F}=\begin{bmatrix} V_1^{*} & T\\ 0 & V_2\end{bmatrix}$ on
$H_1\oplus H_2$, where each $H_j$ is a separable Hilbert space, $V_j\in \mathcal{B}(H_j)$ is an isometry for $j=1,2$, and $T\in \mathcal{B}(H_2,H_1)$.
We examine how the norm and minimum attainment properties of $\EuScript{F}$
depend on the corner entry $T$ and on the defect spaces of the isometries.
The norm formula
$\|\EuScript{F}\|=\frac{\|T\|+\sqrt{\|T\|^{2}+4}}{2}$ holds for every $T$
when $V_1,V_2$ are isometries, and more generally when $V_2$ is a contraction
that is isometric on the range of $T^{*}$. Moreover, $\EuScript{F}$ attains its
norm if and only if $T$ does. We characterize the positive $2\times 2$ block
operator matrices that are absolutely norm attaining and absolutely minimum
attaining, and use this to prove that $\EuScript{F}$ is absolutely norm
attaining if and only if it is absolutely minimum attaining if and only if
$\dim N(V_1^{*})<\infty$ and $T$ has finite rank. The closure of the absolutely
norm attaining class is described by $\dim N(V_1^{*})<\infty$ and compactness
of $T$. Further, we establish that the reduced minimum modulus is
$\gamma(\EuScript{F})=\frac{\sqrt{\|T\|^{2}+4}-\|T\|}{2}=\frac{1}{\|\EuScript{F}\|}$,
and $\EuScript{F}$ attains it if and only if $T$ is norm attaining. We also
show that $\EuScript{F}$ is Fredholm if and only if $\dim N(V_j^{*})<\infty$ for
$j=1,2$, with index $\dim N(V_1^{*})-\dim N(V_2^{*})$, and we give its
Moore--Penrose inverse. Finally, for $V_1=V_2=T_z$ and more generally
$V_j=T_{\theta_j}$ with $\theta_j$ inner, and Toeplitz or Hankel corner entries,
the results become explicit conditions on the symbol.
\end{abstract}
\maketitle
\section{Introduction}

Attainment questions for bounded operators have a long history in operator
theory. Let $H_1,H_2$ be Hilbert spaces and $T\in \clb(H_1,H_2)$. The operator $T$ is called \emph{norm attaining}, denoted by $T\in\mathcal N(H_1,H_2)$, if there is
a unit vector $x\in H_1$ with $\|Tx\|=\|T\|$. Similarly, $T$ is called
\emph{minimum attaining} if there is a unit vector $x\in H_1$ with
$\|Tx\|=m(T)$, where $m(T)=\inf_{\|x\|=1}\|Tx\|$ is the minimum modulus \cite{GindlerTaylor}. The operator $T$ is said to be \emph{reduced minimum attaining} if there is a unit vector
$x\in\ N(T)^{\perp}$ with $\|Tx\|=\gamma(T)$, where
$$
\gamma(T)=\inf\{\|Tx\| : x\in N(T)^{\perp},\ \|x\|=1\}
$$
is the reduced minimum modulus \cite{Apostol}. These properties were studied in
\cite{Carvajal1,Carvajal2,SHK-GR2020}. The questions become considerably more delicate
when one requires the same property to hold for the restriction of $T$ to every nonzero
closed subspace of $H_1$. This leads to the classes of absolutely norm attaining ($\mathcal{AN}$) and absolutely minimum attaining ($\mathcal{AM}$)
operators.

The notions of absolute norm attainment and absolute minimum attainment were introduced by Carvajal and Neves in
\cite{Carvajal1} and \cite{Carvajal2}, respectively. The structure and characterization of positive $\an$-operators are studied in \cite{Paulsen,VR}. The $\mathcal{AN}$ property has also been investigated within the classes of paranormal and hyponormal operators in \cite{GR2018} and \cite{NR2021}, respectively.

Every $\mathcal{AN}$-operator is norm attaining,
and $T\in\mathcal{AN}(H)$ if and only if $T^*T\in\mathcal{AN}(H)_+$
\cite{VR}. Such an operator is a compact perturbation of a scalar multiple of
a partial isometry \cite{VR}. The class $\mathcal{AM}$ has a similar structure but behaves differently. Every $\mathcal{AM}$-operator has closed
range, and the $\mathcal{AM}$-class contains all finite rank operators but excludes
compact operators with infinite dimensional range \cite{Carvajal2}. The
operator norm closures of $\mathcal{AN}$ and $\mathcal{AM}$ classes coincide, and
every operator in this common closure is a compact perturbation of a partial isometry with finite dimensional kernel \cite{AN Closure}.

These classes contain many familiar operators. In particular, compact operators are $\mathcal{AN}$, finite rank operators are
$\mathcal{AM}$, and an orthogonal projection is $\mathcal{AN}$ (equivalently
$\mathcal{AM}$) if and only if its range or kernel is finite dimensional
\cite{Carvajal1,Carvajal2}. Toeplitz and Hankel operators in these classes
have been described in terms of their symbols \cite{AN Toeplitz}. An
idempotent is $\mathcal{AN}$ if and only if it is $\mathcal{AM}$ if and only
if its range or kernel is finite dimensional \cite{AN:Matrix2026}.

These classes are also connected with the invariant subspace problem. The authors in \cite{NR} proved that every non-scalar operator $T$ with
$\sigma_{ess}(T^*T)=\{\alpha\}$ for which $\{\alpha-\lambda:\lambda\in
\sigma_d(T^*T)\}$ is $p$-summable for some $1\le p<\infty$ has a non-trivial
hyperinvariant subspace, and the same holds for normaloid essentially
isometric operators under a summability condition. As consequences,
normaloid $\mathcal{AN}$-operators and minimaloid $\mathcal{AM}$-operators
have non-trivial hyperinvariant subspaces. Moreover, operators of the form $S+K$ with $S$ an isometry and $K$ in the Schatten $p$-class, and operators of the form $\beta V+K$ with $V$ a
partial isometry with finite dimensional kernel and $K$ in the Schatten
$p$-class have non-trivial hyperinvariant subspaces. The same conclusion holds for $\mathcal{AN}$ Toeplitz
operators \cite{AN Toeplitz} and for generalized quadratic operators belonging to the classes $\mathcal{AN}$, $\mathcal{AM}$ or their operator norm closures \cite{AN:Matrix2026}.

For block operator matrices the situation is more delicate: a $2\times2$ block
operator can fail to be $\mathcal{AN}$ or $\mathcal{AM}$ even when its entries
have the corresponding property locally. In particular, for positive $2\times2$ block
matrices the $\mathcal{AN}$ and $\mathcal{AM}$ properties can be described through the diagonal
entries, their essential spectra and a factorization of the off-diagonal entry \cite{AN:Matrix2026}.

A particularly rich family of block operators comes from the \emph{Foguel
operator}
$$
\begin{bmatrix} S^{*} & T \\
0 & S
\end{bmatrix}\quad\text{on } H\oplus H,
$$
where $S$ is the unilateral shift and $T$ is a suitably chosen operator.
Foguel \cite{Foguel} used an operator of this form to give a counterexample to
a problem of Sz.-Nagy: it is power bounded but not similar to a contraction. Halmos \cite{Halmos} then gave a simpler version of
this construction, working directly with block matrices of the above form. Peller \cite{Peller82} analysed the variants in which the corner entry is a Hankel operator in his work on power bounded operators and functional
calculus; see also \cite{Peller}. Pisier's polynomially bounded operator that
is not similar to a contraction \cite{Pisier} is built from operators of the
same type. While the works above concentrate on similarity and growth of
powers, the norm and the modulus $|F|=(F^{*}F)^{\frac{1}{2}}$ of a Foguel operator
$F=\begin{bmatrix} S^{*} & T\\ 0 & S\end{bmatrix}$ were studied by Garcia
\cite{Garcia2009} in 2009, in terms of the corner entry $T$; see also
Raghupathi \cite{Raghupathi2012}, who gave simpler proofs of the norm formula
for $F=\begin{bmatrix} V^{*} & T\\ 0 & V\end{bmatrix}$ on $H\oplus H$, with $V$ an arbitrary
isometry. This suggests that the extremal behaviour of $F$ should be read off
from $T$, and it is this point of view that we pursue. We ask when the norm and
the minimum of such operators are attained, and how this depends on the corner
entry.

In this paper we study the following generalization. Let $H_1,H_2$ be Hilbert
spaces, $V_j\in\clb(H_j)$ isometries, $T\in\clb(H_2,H_1)$, and consider the Foguel-type operator
$$
\Fog=\begin{bmatrix} V_1^{*} & T\\ 0 & V_2\end{bmatrix} \text{ on } H_1\oplus H_2.
$$
The case $V_1=V_2=S$ gives the
classical Foguel operators above. The matrix combines two very different operators, the isometric part, whose defect spaces $N(V_j^{*})$ measure its failure to be unitary, and the arbitrary entry $T$. Our goal is to understand how they determine the extremal  behaviour of $\Fog$. We first show
that
$$
\|\Fog\|=\frac{\|T\|+\sqrt{\|T\|^{2}+4}}{2}
$$
for every $T$ when $V_1,V_2$ are isometries, and more generally when $V_2$ is a
contraction that is isometric on $R(T^{*})$. For $V_1=V_2$ this formula is
contained in \cite{Garcia2009,Raghupathi2012}; here $V_1,V_2$ may act on
different spaces and have different defect spaces. Consequently, $\Fog$
attains its norm if and only if $T$ does. Thus despite
the block structure, norm attainment is decided entirely by the off-diagonal
entry. For the absolute versions we characterize the positive $2\times2$ block
operator matrices in $\mathcal{AN}$-class, which answers a question left open in \cite{AN:Matrix2026}, and
in $\mathcal{AM}$-class. We use this to prove that $\Fog$ is $\mathcal{AN}$ if and only if $\Fog$ is $\mathcal{AM}$ if and only if $\dim N(V_1^{*})<\infty$ and $T$
has finite rank. Replacing finite rank by compactness of $T$ gives the
corresponding description for the closure of the $\mathcal{AN}$-class.

Since $N(\Fog)=N(V_1^{*})\oplus\{0\}$, the minimum modulus $m(\Fog)$ vanishes as soon as $V_1$ is non-unitary, and $\Fog$ is then trivially minimum
attaining. The reduced minimum modulus, however, carries the real information. We prove that 
$$
\gamma(\Fog)=\frac{\sqrt{\|T\|^{2}+4}-\|T\|}{2}=\frac1{\|\Fog\|},
$$
and $\Fog$ attains $\gamma(\Fog)$ if and only if $T$ is norm attaining. In this way
the lower extremal behaviour of $\Fog$ is tied to the upper extremal behaviour of $T$. Further, $\Fog$ always has closed range and is Fredholm precisely when $\dim N(V_j^{*})<\infty$ for $j=1,2$, with
$\mathrm{ind}(\Fog)=\dim N(V_1^{*})-\dim N(V_2^{*})$. We also show that the Moore--Penrose inverse of $\Fog$ is
$$
\Fog^{\dagger}=\begin{bmatrix} V_1 & -V_1TV_2^{*}\\
0 & V_2^{*}
\end{bmatrix},
$$
which is norm attaining exactly when $\Fog$ is reduced minimum attaining, that is,
exactly when $T$ is norm attaining.

We believe these results are of interest for several reasons. First, a question
about a $2\times2$ operator matrix is reduced to a question about the single
operator $T$, while the isometries affect only the kernel, the range and the
index. Second, the condition $\dim N(V_1^{*})<\infty$ with $T$ of
finite rank gives a concrete family of $\mathcal{AN}$ and $\mathcal{AM}$
operators, including non-normal ones. Third, our examples show that the
hypotheses cannot be dropped. A Foguel-type operator $\Fog$ may be Fredholm, norm
attaining and reduced minimum attaining without being $\mathcal{AN}$ or
$\mathcal{AM}$. It may also fail to attain $\gamma(\Fog)$ even when $T$ attains
its minimum modulus. Allowing two different isometries on different spaces also
shows which properties depend only on $V_1$ and which involve $V_2$ as well.
The $\mathcal{AN}$ and $\mathcal{AM}$ conditions and the kernel depend only on
$V_1$, whereas Fredholmness and the index depend on both isometries. This
distinction cannot be seen when $V_1=V_2$. Finally, as the Foguel-type operator
originated in Hardy space theory, our theorems return there. For
$V_1=V_2=T_z$ and more generally $V_j=T_{\theta_j}$ with $\theta_j$ inner, and
Toeplitz or Hankel corner entries, they become explicit conditions on the symbol.

\medskip
The paper is organized as follows. In Section 2, we fix notations and recall the necessary background. In Section 3, we prove the norm formula and show that $\Fog$ attains its norm if and only if $T$ does. We also characterize absolutely norm attaining Foguel-type operators by first establishing a complete characterization
of positive $2\times 2$ block operator matrices belonging to
the $\mathcal{AN}$-class. Section 4 treats the reduced minimum modulus and the absolutely minimum attaining case, including the corresponding characterization of positive $2\times 2$ block operator matrices in the $\mathcal{AM}$-class. Finally, in Section 5, we discuss the Fredholm property, the index and the Moore--Penrose inverse, and apply the results to Toeplitz and Hankel corner entries on the Hardy space.
\section{Preliminaries}

Throughout, all Hilbert spaces are complex, infinite dimensional and separable,
and are denoted by $H,H_1,H_2$. We write $\clb(H_1,H_2)$ for the bounded linear
operators from $H_1$ into $H_2$, and $\clb(H):=\clb(H,H)$ when $H_1=H_2=H$. For
$T\in \clb(H_1,H_2)$, we denote its range and null space by $R(T)$ and
$N(T)$, and its adjoint by $T^{*}\in \clb(H_2,H_1)$. The unit sphere of $H$
is denoted by $S_H=\{x\in H:\|x\|=1\}$. We write $\mathcal F(H_1,H_2)$ and
$\mathcal K(H_1,H_2)$ for the finite rank and the compact operators, with
$\mathcal F(H):=\mathcal F(H,H)$ and $\mathcal K(H):=\mathcal K(H,H)$. The class of self-adjoint and the positive operators on $H$ are denoted by $\clb(H)_{sa}$ and $\clb(H)_{+}$, respectively. 
For $T\in \clb(H_1,H_2)$ the modulus is $|T|=(T^{*}T)^{\frac{1}{2}}$, the unique positive square root of the positive operator $T^*T$.

\smallskip
An operator $T\in \clb(H_1,H_2)$ is said to be \emph{norm attaining} if $\|Tx_0\|=\|T\|$ for
some $x_0\in S_{H_1}$, and the class of such operators is denoted by
$\mathcal N(H_1,H_2)$, with $\mathcal N(H):=\mathcal N(H,H)$. The
\emph{minimum modulus} of $T$ is defined by $m(T)=\inf\{\|Tx\|:x\in S_{H_1}\}$, and $T$ is
\emph{minimum attaining} if $\|Ty_0\|=m(T)$ for some $y_0\in S_{H_1}$. Since
$m(T)=0$ whenever $N(T)\neq\{0\}$, it is natural to look at $T$ on $N(T)^{\perp}$. The \emph{reduced minimum modulus} is defined by
$$
\gamma(T)=\inf\{\|Tx\|:x\in N(T)^{\perp},\ \|x\|=1\},
$$
and $T$ is \emph{reduced minimum attaining} if $\|Tz\|=\gamma(T)$ for some unit
vector $z\in N(T)^{\perp}$. We have $m(T)\le\gamma(T)$, with equality if
$T$ is injective, and $\gamma(T)>0$ if and only if $R(T)$ is closed. We refer to \cite{Apostol} for further details.

\smallskip
An operator $T\in \clb(H_1,H_2)$ is called \emph{absolutely norm attaining}
($\mathcal{AN}$) if the restriction $T|_M : M\to H_2$ is norm attaining for
every nonzero closed subspace $M$ of $H_1$. The collection of all such operators is denoted by $\an(H_1,H_2)$. We refer to \cite{Carvajal1,Paulsen,VR} for further properties and characterizations of this class.

On the other hand, $T\in\clb(H_1,H_2)$ is called \emph{absolutely minimum
attaining} ($\mathcal{AM}$) if $T|_M : M\to H_2$ is minimum attaining for every nonzero closed subspace $M$ of $H_1$. We denote the collection of all such operators by $\am(H_1,H_2)$. We refer to \cite{Carvajal2,AM Cha} for further details on this class.

When $H_1=H_2=H$, we simply write $\mathcal{AN}(H)$ and
$\mathcal{AM}(H)$, respectively. The corresponding subclasses of
positive operators are denoted by $\mathcal{AN}(H)_+$ and
$\mathcal{AM}(H)_+$, respectively.

\smallskip
The direct sum $H_1\oplus H_2$ is a Hilbert space with the inner product
$\langle(u_1,u_2),(v_1,v_2)\rangle=\langle u_1,v_1\rangle_{H_1}+\langle u_2,v_2\rangle_{H_2}$, where $u_j,v_j\in H_j$ for $j=1,2$.
Every $T\in \clb(H_1\oplus H_2)$ has a block representation
$$
T=\begin{bmatrix}
A_{11} & A_{12}\\
A_{21} & A_{22}
\end{bmatrix} \quad \text{with }
A_{ij}=P_iT|_{H_j}\in \clb(H_j,H_i),
$$
where $P_i$ is the orthogonal projection of $H_1\oplus H_2$ onto $H_i$. The
adjoint and products of block operators are computed by the usual matrix rules,
and we write $A_1\oplus A_2$ for $\begin{bmatrix}A_1&0\\0&A_2\end{bmatrix}$. We shall use the following two positivity criteria for
$T=\begin{bmatrix}A_{11}&A_{12}\\A_{12}^{*}&A_{22}\end{bmatrix}$.
First, $T\ge0$ if and only if $A_{11},A_{22}\ge0$ and
$A_{12}=A_{11}^{\frac{1}{2}}CA_{22}^{\frac{1}{2}}$ for some contraction $C$. Second, if
$A_{22}\ge\varepsilon I$ for some $\varepsilon>0$, then $T\ge0$ if and only if
the Schur complement $A_{11}-A_{12}A_{22}^{-1}A_{12}^{*}$ is positive, and the
same holds with the roles of $A_{11}$ and $A_{22}$ exchanged \cite{Osaka}.

Let $A\in \clb(H)$ be a self-adjoint operator. Then there exists a unique pair of positive operators $A^{+}$ and $A^{-}$ with $A^{+}=\frac{|A|+A}{2}$ and $A^{-}=\frac{|A|-A}{2}$ such that $A=A^{+}-A^{-}$. Moreover, $A^{+}A^{-}=A^{-}A^{+}=0$. We refer to \cite[Proposition 4.4.9, Page 158]{Pedersen} for further details.

Let $\alpha\geq 0$ and $(A_j-\alpha I_{H_j})\in \clb(H_j)_{sa}$. We  denote $A_j^{\alpha+}:= (A_j-\alpha I_{H_j})^+$ and $A_j^{\alpha-}:= (A_j-\alpha I_{H_j})^-$. 

\smallskip
For $T\in \clb(H)$, the spectrum is defined by
$\sigma(T)=\{\lambda\in\mathbb C: T-\lambda I\ \text{is not invertible}\}$. The operator $T$ is said to be \emph{Fredholm} if $R(T)$ is closed and both $N(T)$, $N(T^{*})$ are finite dimensional, and its
\emph{index} is given by $\mathrm{ind}(T)=\dim N(T)-\dim N(T^{*})$.
The \emph{essential spectrum} of $T$ is defined by
$$
\sigma_{\mathrm{ess}}(T)=\{\mu\in\mathbb C: T-\mu I\ \text{is not Fredholm}\}.
$$
By Weyl's theorem \cite[Chapter XI, Page 191]{GGK1}, $\sigma_{\mathrm{ess}}(T+K)=\sigma_{\mathrm{ess}}(T)$ for every
$T\in \clb(H)$ and $K\in\mathcal K(H)$.

\smallskip
An operator $V\in \clb(H)$ is called an \emph{isometry} if $V^{*}V=I$. Then $V$ is injective, and has a closed range. Also note that $VV^{*}$ is the orthogonal projection onto $R(V)$, and $I-VV^{*}=P_{N(V^{*})}$, the orthogonal
projection onto $N(V^{*})=R(V)^{\perp}$. The space $N(V^{*})$ is called the \emph{defect space} of $V$, and $V$ is unitary if and only if $N(V^{*})=\{0\}$. The unilateral shift $S$ on $\ell^2(\mathbb N)$, defined by $Se_n=e_{n+1}$ for $n\in\N$, is an isometry with
$N(S^{*})=\C e_1$. The bilateral shift $W$ on $\ell^2(\mathbb{Z})$, defined by $We_n=e_{n+1}$ for $n\in\mathbb{Z}$, is unitary.

Let $T\in \clb(H_1,H_2)$ have a closed range. The \emph{Moore--Penrose inverse} of $T$ is the unique operator $T^{\dagger}\in \clb(H_2,H_1)$ satisfying the four Penrose equations
$$
TT^{\dagger}T=T,\quad T^{\dagger}TT^{\dagger}=T^{\dagger},\quad (TT^{\dagger})^{*}=TT^{\dagger},\quad (T^{\dagger}T)^{*}=T^{\dagger}T.
$$
Also note that $T^{\dagger}$ satisfies the following properties
$$
N(T^{\dagger})=R(T)^{\perp},\quad R(T^{\dagger})=N(T)^{\perp},\quad TT^{\dagger}=P_{R(T)},\quad T^{\dagger}T=P_{N(T)^{\perp}},
$$
where $P_{M}$ denotes the orthogonal projection onto a closed subspace $M$. In particular, for $y\in H_2$, $T^{\dagger}y$ is the unique vector $z\in N(T)^{\perp}$ such that $Tz=P_{R(T)}y$.

If $T$ is invertible, then $T^{\dagger}=T^{-1}$. If $T\neq 0$, then $\gamma(T)=\frac{1}{\|T^{\dagger}\|}$, moreover, $T^{\dagger}$ is norm attaining if and only if $T$ is reduced minimum attaining (see \cite[Theorem 3.2]{SHK-GR2020}). We refer to \cite{Groetsch} for further details on Moore--Penrose inverse.

\section{Norm and Absolutely Norm Attainment}
In this section, we study the norm attainment properties of Foguel-type operators and their relationship with the off-diagonal entry $T$. First we establish a formula for the operator norm under an appropriate condition on the isometries
and prove that norm attainment of the Foguel-type operator is equivalent to norm attainment of $T$. We then establish a complete characterization of positive $2\times 2$ block operator matrices in the $\mathcal{AN}$-class, which enables us to characterize absolutely norm attaining Foguel-type operators.

\begin{proposition}\label{normformula}
Let $T\in \clb(H_2,H_1)$, $V_1\in \clb(H_1)$ be an isometry,
$V_2\in \clb(H_2)$ be a contraction such that $\|V_2x\|=\|x\|$ for all $x\in R(T^*)$. Define the operator $\Fog=\begin{bmatrix}
V_1^* & T\\
0 & V_2
\end{bmatrix}$ on $H_1\oplus H_2$. Then
\begin{equation}\label{neq}
\|\Fog\|=\frac{\|T\|+\sqrt{\|T\|^2+4}}{2}.
\end{equation}
\end{proposition}

\begin{proof}
Let $Q=V_2^*V_2$. Since $V_2$ is a contraction, we have $0\leq Q\leq I_{H_2}$.

Now
$$
\Fog\Fog^*=\begin{bmatrix}
V_1^*V_1+TT^* & TV_2^*\\
V_2T^* & V_2V_2^*
\end{bmatrix}
=\begin{bmatrix}
I_{H_1}+TT^* & TV_2^*\\
V_2T^* & V_2V_2^*
\end{bmatrix}.
$$
If $T=0$, then we get $\|\Fog\|=1$. Assume that $T\neq 0$ and $\lambda>1$. Then
$$
\lambda^2I_{H_1\oplus H_2}-\Fog\Fog^*
=\begin{bmatrix}
(\lambda^2-1)I_{H_1}-TT^* & -TV_2^*\\
-V_2T^* & \lambda^2I_{H_2}-V_2V_2^*
\end{bmatrix}.
$$
Since $\|V_2\|\leq1$ and $\lambda>1$, we have
$\lambda^2I_{H_2}-V_2V_2^*\geq (\lambda^2-1)I_{H_2}>0$.
Therefore, by the Schur complement criterion \cite[Remark 2.4.2, Page 56]{Osaka},
\begin{align}\label{neq1}
\lambda^2I_{H_1\oplus H_2}-\Fog\Fog^*\geq0
\iff\;&
(\lambda^2-1)I_{H_1}-TT^* \notag\\
&\quad
-TV_2^*(\lambda^2I_{H_2}-V_2V_2^*)^{-1}V_2T^*
\geq0.
\end{align}

Since $\left\|\dfrac{V_2V_2^*}{\lambda^2}\right\|
\leq\dfrac{1}{\lambda^2}<1$, the Neumann series gives
$$
(\lambda^2I_{H_2}-V_2V_2^*)^{-1}
=\sum_{j=0}^{\infty}
\frac{(V_2V_2^*)^j}{\lambda^{2(j+1)}}.
$$
Consequently, we get
$$
V_2^*(\lambda^2I_{H_2}-V_2V_2^*)^{-1}V_2=\sum_{j=0}^{\infty}
\frac{V_2^*(V_2V_2^*)^jV_2}{\lambda^{2(j+1)}}=\sum_{j=0}^{\infty}
\frac{Q^{j+1}}{\lambda^{2(j+1)}}=Q(\lambda^2I_{H_2}-Q)^{-1}.
$$
For $y\in H_1$, let $x=T^*y$. Then $x\in R(T^*)$, and hence by the hypothesis, we have $\langle Qx,x\rangle=\|x\|^2$ and as $0\leq Q\leq I_{H_2}$, we get $Qx=x$.
It follows that
$$
(\lambda^2I_{H_2}-Q)x=(\lambda^2-1)x,
$$
and therefore
$$
Q(\lambda^2I_{H_2}-Q)^{-1}x
=\frac{x}{\lambda^2-1}.
$$
Thus for any $y\in H_1$, we obtain
\begin{equation}\label{neq2}
TV_2^*(\lambda^2I_{H_2}-V_2V_2^*)^{-1}V_2T^* y
=\frac{TT^* y}{\lambda^2-1}.
\end{equation}

Substituting \eqref{neq2} into \eqref{neq1}, we obtain
\begin{align*}
\lambda^2I_{H_1\oplus H_2}-\Fog\Fog^*\geq0
&\iff
(\lambda^2-1)I_{H_1}-TT^*
-\frac{TT^*}{\lambda^2-1}\geq0\\
&\iff
(\lambda^2-1)I_{H_1}
-\frac{\lambda^2}{\lambda^2-1}TT^*
\geq0\\
&\iff
TT^*
\leq
\frac{(\lambda^2-1)^2}{\lambda^2}I_{H_1}\\
&\iff
\|T\|
\leq
\frac{\lambda^2-1}{\lambda}.
\end{align*}
Hence
\begin{equation}\label{neq3}
\|\Fog\|\leq \lambda
\iff
\|T\|\leq\frac{\lambda^2-1}{\lambda},
\qquad \lambda>1.
\end{equation}
Thus we obtain $\|\Fog\|=\inf\left\{\lambda>1:
\|T\|\leq\frac{\lambda^2-1}{\lambda}\right\}$.
The inequality in \eqref{neq3} is equivalent to $\lambda^2-\|T\|\lambda-1\geq 0$.
The positive root of the equation $\lambda^2-\|T\|\lambda-1=0$ is
$\lambda_0=\dfrac{\|T\|+\sqrt{\|T\|^2+4}}{2}$. Consequently, it follows that
$$
\|\Fog\|=\lambda_0=\frac{\|T\|+\sqrt{\|T\|^2+4}}{2}.
$$
This completes the proof.
\end{proof}

\begin{remark}
(i) For $V_1=V_2=S$, the unilateral shift on
$\ell^2_E(\mathbb N)$ with $E$ an arbitrary separable complex Hilbert space, the norm
formula \eqref{neq} is due to Garcia \cite{Garcia2009}. For $V_1=V_2=V$ an
arbitrary isometry on a single Hilbert space, the formula was proved by Raghupathi
\cite{Raghupathi2012}. Proposition \ref{normformula} allows $V_1$ and $V_2$ to act on
different spaces, and allows $V_2$ to be merely a contraction that is isometric
on $R(T^{*})$.

(ii) In particular, if $V_j\in \clb(H_j)$ is an isometry for $j=1,2$, then
$\|V_2x\|=\|x\|$ for all $x\in H_2$. Thus the norm formula holds for every
$T\in \clb(H_2,H_1)$.
\end{remark}

We recall the following results, which will be useful in our subsequent analysis of the norm attaining properties of Foguel-type operators (see \cite[Theorem 2.5]{Paulsen}).

\begin{lemma}\label{lem:PP}
Let $T\in\mathcal{B}(\mathcal{H}_1,\mathcal{H}_2)$. Then the following are equivalent:
\begin{enumerate}
\item $T\in\mathcal{N}(H_1,H_2)$;
\item $T^*\in\mathcal{N}(H_2,H_1)$;
\item $T^*T\in\mathcal{N}(H_1)$;
\item $TT^*\in\mathcal{N}(H_2)$;
\item $\|T\|^2$ is an eigenvalue of $T^*T$.
\end{enumerate}
\end{lemma}

\begin{theorem}\label{NA Foguel}
Let $V_j\in\clb(H_j)$ be an isometry for $j=1,2$, and $T\in\clb(H_2,H_1)$. Consider the Foguel-type operator
$\Fog=\begin{bmatrix}
V_1^* & T\\
0 & V_2
\end{bmatrix}\in\clb(H_1\oplus H_2)$. Then $\Fog\in\na(H_1\oplus H_2)$ if and only if $T\in\na(H_2,H_1)$.
\end{theorem}

\begin{proof}
If $T=0$, then $\Fog=
\begin{bmatrix}
V_1^* & 0\\
0 & V_2
\end{bmatrix}$, and hence $\|\Fog\|=1$. Since $V_j$ is an isometry, $\Fog$
attains its norm. Thus we may assume that $T\neq0$.

Let $\lambda=\|\Fog\|^2$. By the norm formula $\|\Fog\|=
\frac{\|T\|+\sqrt{\|T\|^2+4}}{2}$, we can write $\|T\|=
\|\Fog\|-\dfrac{1}{\|\Fog\|}
=\sqrt{\lambda}-\dfrac{1}{\sqrt{\lambda}}$.

\smallskip
Assume that $\Fog$ is norm attaining. Then there exists a
unit vector $(x,y)\in H_1\oplus H_2$ such that $\Fog^*\Fog(x,y)=
\lambda(x,y)$. Since $V_2^*V_2=I_{H_2}$, we get
$$
\Fog^*\Fog
=\begin{bmatrix}
V_1V_1^* & V_1T\\
T^*V_1^* & T^*T+I_{H_2}
\end{bmatrix}.
$$
Thus it follows that
\begin{align}
V_1V_1^*x + V_1Ty&=\lambda x, \label{na1}\\
T^*V_1^*x+(T^*T+I_{H_2})y&=\lambda y. \label{na2}
\end{align}
Applying $V_1^*$ to \eqref{na1}, we obtain $V_1^*x+Ty=\lambda V_1^*x$, and hence $Ty=(\lambda-1)V_1^*x$. Consequently, $T^*Ty=(\lambda-1)T^*V_1^*x$, that is, $T^*V_1^*x
=\dfrac{1}{\lambda-1}T^*Ty$. Substituting this into \eqref{na2}, we obtain $T^*Ty=
\dfrac{(\lambda-1)^2}{\lambda}y=\left(\sqrt{\lambda}-\dfrac{1}{\sqrt{\lambda}}\right)^2 y=\|T\|^2 y$.

Since $(x,y)\neq(0,0)$ and $\lambda>1$, we must have $y\neq0$, otherwise
\eqref{na1} would imply $x=0$. Thus $\|T\|^2$ is an
eigenvalue of $T^*T$, and consequently $T$ attains its norm.

\smallskip
Conversely, let $T$ be norm attaining. Then there exists a
unit vector $y\in H_2$ such that $T^*Ty=\|T\|^2y$. Choose $x:=\dfrac{1}{\lambda-1}V_1Ty$. Then $x$ is a nonzero vector in $H_1$.

Now $\Fog^*\Fog(x,y)=(V_1V_1^*x+V_1Ty, T^*V_1^*x+T^*Ty+y)$. From the first component, we get 
\begin{align*}
V_1V_1^*x+V_1Ty&=\frac{1}{\lambda-1}V_1V_1^*V_1Ty+V_1Ty \\
&=\frac{1}{\lambda-1}V_1Ty+V_1Ty \\
&=\frac{\lambda}{\lambda-1}V_1Ty \\
&=\lambda x.
\end{align*}
From the second component, we get
\begin{align*}
T^*V_1^*x+T^*Ty+y &=\frac{T^*Ty}{\lambda-1}+T^*Ty+y \\
&=\left(\frac{\|T\|^2}{\lambda-1}
+\|T\|^2+1
\right)y \\
&=\left(\frac{\|T\|^2\lambda}{\lambda-1}+1\right)y \\
&=\lambda y.
\end{align*}
Therefore, $\Fog^*\Fog(x,y)=\lambda(x,y)$ with $(x,y)\neq (0,0)$, and so $\|\Fog\|^2=\lambda$ is an eigenvalue for $\Fog^*\Fog$. Thus $\Fog$ is norm attaining.  
\end{proof}

\medskip
Our next goal is to discuss the absolutely norm attaining property of the operator given in \cite[Question 3.12]{AN:Matrix2026}. A sufficient condition for a positive $2\times 2$ block operator matrix to be $\mathcal{AN}$ is discussed
in \cite{AN:Matrix2026}. Here we obtain a complete characterization. This result serves as a key tool in our subsequent characterization of Foguel-type operators to be $\an$.

\begin{theorem}\label{AN block iff}
Let $T=\begin{bmatrix}
A_1 & X\\
X^* & A_2
\end{bmatrix}\in \clb(H_1\oplus H_2)_+$. Then the following are equivalent:
\begin{enumerate}
\item[(i)] $T\in \an(H_1\oplus H_2)$;

\item[(ii)] $A_j\in \an(H_j)_+$ with
$\es(A_1)=\es(A_2)=\{\alpha\}$
for some $\alpha\geq 0$, and
$$
X=
\left(A_1^{\alpha+}+F_1\right)^{\frac{1}{2}}
C
\left(A_2^{\alpha+}+F_2\right)^{\frac{1}{2}} + F
$$
for some contraction $C$, $F\in\mathcal{F}(H_2,H_1)$ and $F_j\in\mathcal{F}(H_j)_{sa}$
such that $A_j^{\alpha+}+F_j\in \clk(H_j)_+$.
\end{enumerate}
\end{theorem}

\begin{proof}
(i)$\implies$(ii): Assume that $T\in \an(H_1\oplus H_2)$. Then by \cite[Theorem 2.5]{VR}, there
exists $\alpha\geq 0$, $F\in\mathcal{F}(H_1\oplus H_2)_+$ and
$K\in \clk(H_1\oplus H_2)_+$ such that $T=\alpha I_{H_1\oplus H_2}+K-F$.

Suppose
$$
K=
\begin{bmatrix}
K_{11} & K_{12}\\
K_{21} & K_{22}
\end{bmatrix}
\qquad\text{and}\qquad
F=
\begin{bmatrix}
F_{11} & F_{12}\\
F_{21} & F_{22}
\end{bmatrix}.
$$
Since $K\geq 0$, we have $K_{21}=K_{12}^*$,
$K_{11},K_{22}\geq 0$ and
$K_{12}=K_{11}^{\frac{1}{2}} C K_{22}^{\frac{1}{2}}$ for some contraction $C$.
Since $F\geq 0$, we have $F_{21}=F_{12}^*$,
$F_{11},F_{22}\geq 0$. Then
$$
T-\alpha I=
\begin{bmatrix}
K_{11} & K_{12}\\
K_{12}^* & K_{22}
\end{bmatrix}
-
\begin{bmatrix}
F_{11} & F_{12}\\
F_{12}^* & F_{22}
\end{bmatrix}
=
\begin{bmatrix}
K_{11}-F_{11} & K_{12}-F_{12}\\
K_{12}^*-F_{12}^* & K_{22}-F_{22}
\end{bmatrix},
$$
that is,
$$
\begin{bmatrix}
A_1-\alpha I_{H_1} & X\\
X^* & A_2-\alpha I_{H_2}
\end{bmatrix}
=
\begin{bmatrix}
K_{11}-F_{11} & K_{12}-F_{12}\\
K_{12}^*-F_{12}^* & K_{22}-F_{22}
\end{bmatrix}.
$$
Therefore, $A_j=\alpha I_{H_j}+K_{jj}-F_{jj}$ and
$X=K_{12}-F_{12}$. Thus by \cite[Theorem 5.1]{Paulsen}, we get $A_j\in \an(H_j)_+$. Also by Weyl's theorem, $\es(A_1)=\es(A_2)=\{\alpha\}$.

By \cite[Theorem 4.1]{AN Closure}, we can write
$$
A_j-\alpha I_{H_j}
=(A_j-\alpha I_{H_j})^+
-(A_j-\alpha I_{H_j})^-,
$$
where $(A_j-\alpha I_{H_j})^+\in \clk(H_j)_+$ and
$(A_j-\alpha I_{H_j})^-\in\clf(H_j)_+$. Hence
$$
K_{jj}-F_{jj}
=(A_j-\alpha I_{H_j})^+
-(A_j-\alpha I_{H_j})^-.
$$
Therefore
$$
K_{jj}
=(A_j-\alpha I_{H_j})^+
+\left(F_{jj}-(A_j-\alpha I_{H_j})^-\right).
$$
Since both $F_{jj}$ and $(A_j-\alpha I_{H_j})^-$ are finite rank
self-adjoint operators, the operator $F_j:=F_{jj}-(A_j-\alpha I_{H_j})^-$ is finite rank self-adjoint. Thus
$$
K_{jj}=(A_j-\alpha I_{H_j})^+ +F_j.
$$

Substituting into the formula for $X$, we obtain
$$
X=\left((A_1-\alpha I_{H_1})^+ +F_1\right)^{\frac{1}{2}}\,C\,\left((A_2-\alpha I_{H_2})^+ +F_2\right)^{\frac{1}{2}}
-F_{12}.
$$
Writing $A_j^{\alpha+}=(A_j-\alpha I_{H_j})^+$ and $F:=-F_{12}$, we get
$$
X=(A_1^{\alpha+}+F_1)^{\frac{1}{2}}\,C\,
(A_2^{\alpha+}+F_2)^{\frac{1}{2}}+F.
$$
\smallskip

(ii)$\implies$(i): Let (ii) hold true. Since $A_j\in \an(H_j)_+$ with
$\es(A_j)=\{\alpha\}$, \cite[Theorem 4.1]{AN Closure} yields
$$
A_j=\alpha I_{H_j}
+(A_j-\alpha I_{H_j})^+
-(A_j-\alpha I_{H_j})^-,
$$
where $A_j^{\alpha+}=(A_j-\alpha I_{H_j})^+\in \clk(H_j)_+$ and $A_j^{\alpha-}=(A_j-\alpha I_{H_j})^-\in\mathcal{F}(H_j)_+$.

By the assumption, there exist a contraction $C$, $F\in \clf(H_2,H_1)$ and $F_j\in \clf(H_j)_{sa}$ with
$A_j^{\alpha+}+F_j\in \clk(H_j)_+$ such that
$$
X=(A_1^{\alpha+}+F_1)^{\frac{1}{2}}\,C\,
(A_2^{\alpha+}+F_2)^{\frac{1}{2}}+F.
$$
Consider $\widetilde{K}_j=A_j^{\alpha+} + F_j$. Then using the block positivity criterion,
$$
\widetilde{K}=
\begin{bmatrix}
\widetilde{K}_1 & \widetilde{K}_1^{\frac{1}{2}}C\widetilde{K}_2^{\frac{1}{2}}\\
\widetilde{K}_2^{\frac{1}{2}}C^*\widetilde{K}_1^{\frac{1}{2}} & \widetilde{K}_2
\end{bmatrix}
\geq 0,
$$
and
$$
\widetilde{F}=
\begin{bmatrix}
-A_1^{\alpha-}-F_1 & F\\
F^* & -A_2^{\alpha-}-F_2
\end{bmatrix}
$$
is a self-adjoint finite rank operator. Consequently, we obtain $T=\alpha I_{H_1\oplus H_2}+\widetilde{K}+\widetilde{F}$, where $\widetilde{K}$ is positive compact and $\widetilde{F}$ is self-adjoint finite rank. Hence by \cite[Theorem 5.1]{Paulsen}, it follows that $T\in \an(H_1\oplus H_2)$. 
\end{proof}

Next, we characterize the Foguel-type operators that are absolutely norm attaining.
\begin{theorem}\label{AN Foguel}
Let $V_j\in\clb(H_j)$ be an isometry for $j=1,2$, and $T\in\clb(H_2,H_1)$.
Consider the Foguel-type operator $\Fog=\begin{bmatrix}
V_1^* & T\\
0 & V_2
\end{bmatrix}$ on $H_1\oplus H_2$.
Then the following are equivalent:
\begin{enumerate}
    \item $\Fog\in\an(H_1\oplus H_2)$;
    
    \item $\dim N(V_1^*)<\infty$ and $T$ is of finite rank.
\end{enumerate}
\end{theorem}

\begin{proof}
Since $V_2$ is an isometry, we have $$\Fog^*\Fog=
\begin{bmatrix}
V_1V_1^* & V_1T\\
T^*V_1^* & T^*T+I_{H_2}
\end{bmatrix}.$$
Let $A_1=V_1V_1^*, A_2=T^*T+I_{H_2}$ and $X=V_1T$. Note that $V_1V_1^*$ is an orthogonal projection onto $R(V_1)$. Consequently, we can write $V_1V_1^*=I-P_{N(V_1^*)}$, where $P_{N(V_1^*)}$ denotes the orthogonal projection onto
$N(V_1^*)$. Hence
$$
\Fog^*\Fog
=I_{H_1\oplus H_2}+\begin{bmatrix}
-P_{N(V_1^*)} & V_1T\\
T^*V_1^* & T^*T
\end{bmatrix}.
$$

(i)$\implies$(ii): Assume that (i) holds true. Then by \cite[Corollary 2.11]{VR}, $\Fog^*\Fog\in\an(H_1\oplus H_2)_+$. By Theorem \ref{AN block iff}, it follows that $A_1$ and $A_2$ are $\an$ with $\es(A_1)=\es(A_2)$ and $X$ is compact. Consequently, we get $T=V_1^*X$ is compact. Thus by Weyl's theorem, $\es(A_2)=\{1\}=\es(A_1)$. If $N(V_1^*)$ is infinite dimensional, then $0$ is an eigenvalue of $V_1V_1^*$ with infinite multiplicity. In this case, $0$ would belong to $\es(A_1)$, which is impossible. Hence $N(V_1^*)$ must be finite dimensional.

Moreover, by Theorem \ref{AN block iff}, we have $X=(A_1^{\alpha+}+F_1)^{\frac{1}{2}}\,C\,(A_2^{\alpha+}+F_2)^{\frac{1}{2}}+F$, where $\alpha=1$, $F_j\in\clf(H_j)_{sa}$, $F\in\clf(H_2,H_1)$ and $C$ is a contraction.

Since $\alpha=1$, we get $A_1^{\alpha+}=(A_1-\alpha I_{H_1})^+=(I_{H_1}-P_{N(V_1^*)} -I_{H_1})^+=0$. Thus $X$ becomes a finite rank operator. Since $V_1$ is an isometry, $X=V_1T$ has finite rank implies $T$ has finite rank.

\smallskip
(ii)$\implies$(i): Assume that $\dim N(V_1^*)<\infty$ and $T$ has finite rank. Then $\Fog^*\Fog=I_{H_1\oplus H_2} + F$, where $F=\begin{bmatrix}
-P_{N(V_1^*)} & V_1T\\
T^*V_1^* & T^*T
\end{bmatrix}$ is a finite rank operator. By \cite[Proposition 4.5]{Paulsen} and \cite[Corollary 2.11]{VR}, we obtain (i).
\end{proof}

Next, we discuss when $\Fog\in\overline{\an(H_1\oplus H_2)}$, where the closure
is taken with respect to the operator norm. The closure of the
$\mathcal{AN}$-class is described in detail in \cite{AN Closure}, and a characterization of positive $2\times 2$ block operator matrices in this closure is given in \cite{AN:Matrix2026}. Now we describe the same for Foguel-type operators.

\begin{theorem}\label{closure-AN-Foguel}
Let $V_j\in\clb(H_j)$ be an isometry for $j=1,2$, and $T\in\clb(H_2,H_1)$.
Consider the Foguel-type operator
$\Fog=\begin{bmatrix}
V_1^* & T\\
0 & V_2
\end{bmatrix}$ on $H_1\oplus H_2$.
Then the following are equivalent:
\begin{enumerate}
    \item $\Fog\in
    \overline{\mathcal{AN}(H_1\oplus H_2)}$;
    
    \item $\dim N(V_1^*)<\infty$ and $T$ is compact.
\end{enumerate}
\end{theorem}

\begin{proof}
Since $V_j$ is isometry for $j=1,2$, we have $V_2^*V_2=I_{H_2}$ and $V_1V_1^*=I_{H_1}-P_{N(V_1^*)}$,
where $P_{N(V_1^*)}$ denotes the orthogonal projection
onto $N(V_1^*)$. Consequently,
$$
\Fog^*\Fog=\begin{bmatrix}
V_1V_1^* & V_1T\\
T^*V_1^* & T^*T+I_{H_2}
\end{bmatrix}=\begin{bmatrix}
A_1 & X\\
X^* & A_2
\end{bmatrix},
$$
where $A_1=V_1V_1^*$, $A_2=T^*T+I_{H_2}$ and $X=V_1T$.

\smallskip
(i)$\implies$(ii):
Assume that $\Fog\in\overline{\mathcal{AN}(H_1\oplus H_2)}$. Then by \cite[Theorem 3.15]{AN Closure}, $\Fog^*\Fog \in\overline{\mathcal{AN}(H_1\oplus H_2)}_{+}$. Hence by \cite[Theorem 5.3]{AN:Matrix2026}, we have $A_j\in \overline{\an(H_j)}_+$ with $\es(A_1)=\es(A_2)$ and $X$ is compact.
Since $V_1^*V_1=I_{H_1}$, we obtain $T=V_1^*(V_1T)=V_1^*X\in\clk(H_2,H_1)$.
It follows that $T^*T$ is compact. Hence by Weyl's theorem, $\es(A_2)=\es(T^*T+I_{H_2})=\{1\}$. Since $\es(A_1)
=\es(A_2)$, we have $\es(V_1V_1^*)=\{1\}$.

If $\dim N(V_1^*)=\infty$, then $0$ is an eigenvalue of $V_1V_1^*=I_{H_1}-P_{N(V_1^*)}$ of infinite multiplicity.
Consequently,
$0\in\es(V_1V_1^*)$, which contradicts
$\es(V_1V_1^*)=\{1\}$.
Therefore, $\dim N(V_1^*)<\infty$.

\medskip
(ii)$\implies$(i):
Conversely, let (ii) hold true.
Then $P_{N(V_1^*)}$ has finite rank. Thus by
Weyl's theorem, $\es(A_1)=\es(I_{H_1}-P_{N(V_1^*)})
=\{1\}$.
Similarly, since $T^*T$ is compact, we obtain $\es(A_2)=\es(T^*T+I_{H_2})=\{1\}$. By \cite[Theorem 4.6]{AN Closure}, it follows that $A_j\in\overline{\an(H_j)}_+$.
Moreover, $X=V_1T$ is compact, as
$T$ is compact. Therefore, \cite[Theorem 5.3]{AN:Matrix2026} implies that $\Fog^*\Fog\in \overline{\an(H_1\oplus H_2)}_+$. Finally, by \cite[Theorem 3.15]{AN Closure}, we conclude that $\Fog\in\overline{\an(H_1\oplus H_2)}$.
\end{proof}

\section{Minimum and Absolutely Minimum Attainment}
In this section, we examine the minimum modulus and reduced minimum modulus of Foguel-type operators. We first determine
the minimum modulus in the non-unitary case and derive an
explicit formula for the reduced minimum modulus. We then
characterize reduced minimum attainment in terms of norm
attainment of the off-diagonal operator $T$. Finally, by
establishing a characterization of absolutely minimum
attaining positive $2\times2$ block operator matrices, we
obtain necessary and sufficient conditions for a Foguel
type operator to be absolutely minimum attaining.
\begin{proposition}
    Let $V_1\in \clb(H_1)$ be a non-unitary isometry, $V_2\in\clb(H_2)$ is an isometry and $T\in \clb(H_2,H_1)$. Consider the Foguel-type operator
$\Fog=\begin{bmatrix}
V_1^* & T\\
0 & V_2
\end{bmatrix}$.  Then $m(\Fog)=0$.
\end{proposition}

\begin{proof}
Let $(x,y)\in N(\Fog)$. Then $\Fog(x,y)=(0,0)$ gives
$$
V_1^* x + Ty = 0\, \text{ and } \,V_2y=0.
$$
Since $N(V_2)=\{0\}$, we obtain $y=0$. Consequently, $V_1^{*}x=0$,
and therefore
$$
N(\Fog)=N(V_1^{*})\oplus\{0\}.
$$
Since $V_1$ is non-unitary, $N(V_1^{*})\neq\{0\}$. Choose $x\in N(V_1^{*})$ with $\|x\|=1$. Then $\|(x,0)\|=1$
and
$$
\Fog(x,0)=(V_1^*x,0)=(0,0).
$$
Hence $m(\Fog)=\inf_{\|(x,y)\|=1}\|\Fog(x,y)\|=0$.
\end{proof}

For a quadratic operator $\begin{bmatrix}
a I_{H_1} & X \\
0 & b I_{H_2}
\end{bmatrix}$ on $H_1\oplus H_2$, the norm was computed in \cite[Lemma 1.2]{Feldman}, and the minimum
modulus of such an operator was obtained in \cite[Theorem 3.1]{Bhuia}, where
it is derived from the inverse and the norm formula of \cite{Feldman}. In the
following lemma we give a direct proof of the formula, in a form that also covers the cases $a=0$ and $b=0$. This representation is useful in determining the reduced minimum modulus of Foguel-type operators.

\begin{lemma}\label{minmod:quadratic}
Let $T=\begin{bmatrix}
aI_{H_1} & X\\
0 & bI_{H_2}
\end{bmatrix}$, where $a,b\in\C$ and $X\in\clb(H_2,H_1)$. Then the minimum modulus of $T$ is given by
\begin{equation}\label{minmod}
m(T)=\frac{1}{2}
\left[
\left(\|X\|^2+(|a|+|b|)^2\right)^{\frac{1}{2}}
-\left(\|X\|^2+(|a|-|b|)^2\right)^{\frac{1}{2}}
\right].
\end{equation}
\end{lemma}

\begin{proof}
If $a=0$, then for any $x\in H_1$, $T(x,0)=(0,0)$, and hence $m(T)=0$. Similarly, if $b=0$, then $m(T)=0$. In both cases, the right hand side of \eqref{minmod} is also equal to $0$. Thus we may assume that $a\neq 0$ and $b\neq 0$.

If $X=0$, then $T=aI_{H_1}\oplus bI_{H_2}$. In this case, $m(T)=\min\{|a|,|b|\}$, which agrees with \eqref{minmod}.

Next, assume that $X\neq0$. Let $t>0$. Then
$$
T^*T-t^2I_{H_1\oplus H_2}=
\begin{bmatrix}
(|a|^2-t^2)I_{H_1} & \overline{a}X\\
aX^* & (|b|^2-t^2)I_{H_2}+X^*X
\end{bmatrix}.
$$
Suppose that $0<t<\min\{|a|,|b|\}$. Then $(|a|^2-t^2)I_{H_1}$ is positive and invertible. Hence by the Schur complement criterion \cite[Theorem 2.4.1, Page 55]{Osaka},
$$
T^*T-t^2I_{H_1\oplus H_2}\geq0 \iff (|b|^2-t^2)I_{H_2}+X^*X-\frac{|a|^2}{|a|^2-t^2}X^*X\geq0.
$$
Equivalently, \begin{equation}\label{scinq}
(|a|^2-t^2)(|b|^2-t^2)I_{H_2}-t^2X^*X\geq0.
\end{equation}
Since $X^*X$ is a positive operator, \eqref{scinq} holds if and only if 
\begin{equation}\label{scinq1}
(|a|^2-t^2)(|b|^2-t^2)\geq t^2\|X\|^2.
\end{equation}
Since $m(T)^2=\inf\sigma(T^*T)$ by \cite[Theorem 1]{Bouldin}, we obtain
$$
T^*T-t^2I_{H_1\oplus H_2}\geq0
\quad\iff\quad
t\leq m(T).
$$
It is easy to see that $m(T)\le\min\{|a|,|b|\}$. Indeed, for a unit vector
$x\in H_1$ we have $T(x,0)=(ax,0)$, so $m(T)\le|a|$. For a unit vector $y\in H_2$, choose $w=(-a^{-1}Xy,\,y)$. Then $T w=(0,by)$ and $\|w\|\ge\|y\|=1$,
hence $m(T)\le \|Tw\|/\|w\|\le|b|$.

\smallskip
\noindent If $m(T)=\min\{|a|,|b|\}$, then $(|a|^2-m(T)^2)(|b|^2-m(T)^2)=0$, while $m(T)^2\|X\|^2>0$, which contradicts \eqref{scinq1}. Thus $m(T)<\min\{|a|,|b|\}$.

For $0<t\leq m(T)$, define
$$
g(t):=(|a|^2-t^2)(|b|^2-t^2)-t^2\|X\|^2.
$$
Since $m(T)$ is the largest value of $t$ for which $g(t)\geq0$ and $g$ is strictly decreasing continuous function on $(0,m(T)]$, it follows that $g(m(T))=0$. Then we obtain
$$
m(T)^4-\left(|a|^2+|b|^2+\|X\|^2\right)m(T)^2+|a|^2|b|^2
=0.
$$
Therefore,
$$
m(T)^2=\frac{1}{2}\left[|a|^2+|b|^2+\|X\|^2\pm
\sqrt{\left(|a|^2+|b|^2+\|X\|^2\right)^2-
4|a|^2|b|^2}\right].
$$
Since $m(T)^2<\min\{|a|^2,|b|^2\}$, we take the smaller root. Also observe that 
$$
\left(\|X\|^2+(|a|+|b|)^2\right)
\left(\|X\|^2+(|a|-|b|)^2\right)
=\left(\|X\|^2+|a|^2+|b|^2\right)^2-4|a|^2|b|^2.
$$
Hence 
$$
m(T)^2
=
\frac{1}{4}
\left[
\left(\|X\|^2+(|a|+|b|)^2\right)^{\frac{1}{2}}
-\left(\|X\|^2+(|a|-|b|)^2\right)^{\frac{1}{2}}
\right]^2.
$$
Since $m(T)\geq0$, taking the positive square root, we obtain
$$
m(T)=\frac{1}{2}
\left[\left(\|X\|^2+(|a|+|b|)^2\right)^{\frac{1}{2}}
-\left(\|X\|^2+(|a|-|b|)^2\right)^{\frac{1}{2}}\right].
$$
This proves the result.
\end{proof}

\begin{proposition}\label{reduced minmod}
Let $V_j\in\clb(H_j)$ be an isometry for $j=1,2$, and  $T\in\clb(H_2,H_1)$. Then the reduced minimum modulus of $\Fog=
\begin{bmatrix}
V_1^* & T\\
0 & V_2
\end{bmatrix}$ is given by
\begin{equation}\label{rm:Fog}
\gamma(\Fog)=
\frac{\sqrt{\|T\|^2+4}-\|T\|}{2}.
\end{equation}
\end{proposition}

\begin{proof}
We have $N(\Fog)=N(V_1^*)\oplus\{0\}.$ Consequently,
$N(\Fog)^\perp=N(V_1^*)^\perp\oplus H=\overline{R(V_1)}\oplus H_2$. Since $V_1$ is an isometry, $R(V_1)$ is closed. Thus $N(\Fog)^\perp=R(V_1)\oplus H_2$.

\smallskip
Now define
$U:H_1\oplus H_2\longrightarrow R(V_1)\oplus H_2$
by $U(x,y)=(V_1x,y)$ for $(x,y)\in H_1\oplus H_2$. Observe that $\|U(x,y)\|^2=\|V_1x\|^2+\|y\|^2=\|x\|^2+\|y\|^2$ for every $(x,y)\in H_1\oplus H_2$. Thus $U$ is an isometry. Moreover, if $(z,y)\in R(V_1)\oplus H_2$, then $z=V_1x$ for
some $x\in H_1$. Therefore, $U(x,y)=(z,y)$, and so $U$ is onto. Hence $U$ is unitary. 

For $(x,y)\in H_1\oplus H_2$, we have
$$
\Fog U(x,y)=\Fog(V_1x,y)=\bigl(V_1^*V_1x+Ty,V_2y\bigr)=(x+Ty,V_2y).
$$
Therefore, we can write $\Fog U=
\begin{bmatrix}
I_{H_1} & T\\
0 & V_2
\end{bmatrix}$. Since $U$ is unitary, the above identity gives
$$
\gamma(\Fog)=
\inf_{\substack{z\in N(\Fog)^\perp,\\\|z\|=1}}
\|\Fog z\|=\inf_{\substack{(x,y)\in H_1\oplus H_2,\\
\|(x,y)\|=1}}
\|\Fog U(x,y)\|=m\left(\begin{bmatrix}
    I_{H_1} & T \\
    0 & V_2
\end{bmatrix}\right).
$$
Observe that $\begin{bmatrix}
    I_{H_1} & T \\
    0 & V_2
\end{bmatrix}=\begin{bmatrix}
    I_{H_1} & 0 \\
    0 & V_2
\end{bmatrix}\begin{bmatrix}
    I_{H_1} & T \\
    0 & I_{H_2}
\end{bmatrix}$, where $\begin{bmatrix}
    I_{H_1} & 0 \\
    0 & V_2
\end{bmatrix}$ is an isometry. Consequently, we get $m\left(\begin{bmatrix}
    I_{H_1} & T \\
    0 & V_{_2}
\end{bmatrix}\right)=m\left(\begin{bmatrix}
    I_{H_1} & T \\
    0 & I_{H_2}
\end{bmatrix}\right)$. By Lemma \ref{minmod:quadratic}, it follows that $\gamma(\Fog)=
\dfrac{\sqrt{\|T\|^2+4}-\|T\|}{2}$. This completes the proof.  
\end{proof}

\begin{remark}\label{closed range and unitary V_1}
Let $V_j\in\clb(H_j)$ be an isometry for $j=1,2$, and $T\in\clb(H_2,H_1)$. Then $\Fog=\begin{bmatrix}
    V_1^{*} & T \\
    0 & V_2
\end{bmatrix}$ always has closed range. If $V_1$ is unitary, then $m(\Fog)=\gamma(\Fog)$.
Moreover, $\Fog$ is non-injective if and only if
$V_1$ is non-unitary. In this case, $\Fog$ is always minimum attaining.
\end{remark}

\begin{corollary}
 Let $\Fog$ be the Foguel-type operator as in Proposition \ref{reduced minmod}. Then the reduced minimum modulus of $\Fog$ is given by $\gamma(\Fog)=\dfrac{1}{\|\Fog\|}$.
\end{corollary}

\begin{proof}
Since $V_1$ and $V_2$ are isometries, $\|V_2x\|=\|x\|$ for every $x\in H_2$. Thus
Proposition \ref{normformula} gives
$$
\|\Fog\|=\frac{\|T\|+\sqrt{\|T\|^2+4}}{2}.
$$
Multiplying this by the expression for $\gamma(\Fog)$ in \eqref{rm:Fog} yields
$$
\|\Fog\|\,\gamma(\Fog)
=\frac{\|T\|+\sqrt{\|T\|^2+4}}{2}\cdot\frac{\sqrt{\|T\|^2+4}-\|T\|}{2}
=\frac{(\|T\|^2+4)-\|T\|^2}{4}=1.
$$
Hence $\gamma(\Fog)=\dfrac{1}{\|\Fog\|}$.
\end{proof}

Next, we give the characterization of reduced minimum attainment of Foguel-type operators. Recall that an operator $A\in \clb(H_1,H_2)$ is said to be \emph{reduced minimum attaining} if there exists a unit vector $z\in N(A)^{\perp}$ such that $\|Az\|=\gamma(A)$.

\begin{theorem}\label{thm:RMA}
Let $V_j\in \clb(H_j)$ be an isometry for $j=1,2$, and $T\in \clb(H_2,H_1)$. Consider the Foguel-type operator $\Fog=\begin{bmatrix}
V_1^* & T \\
0 & V_2
\end{bmatrix}$ on $H_1\oplus H_2$.
Then the following are equivalent:
\begin{enumerate}
\item $\Fog$ attains its reduced minimum modulus;
\item $T\in\mathcal N(H_2,H_1)$.
\end{enumerate}
\end{theorem}

\begin{proof}
As in Proposition \ref{reduced minmod}, $N(\Fog)^\perp=R(V_1)\oplus H_2$, and
$U : H_1\oplus H_2\to R(V_1)\oplus H_2$ given by $U(x,y)=(V_1x,y)$ is unitary. For $(x,y)\in H_1\oplus H_2$,
$$
\Fog U(x,y)=(V_1^*V_1x+Ty,\,V_2y)=(x+Ty,\,V_2y),
$$
that is,
$$
\Fog U=\begin{bmatrix}
I_{H_1} & 0\\
0 & V_2
\end{bmatrix}G, \text{ where } G:=\begin{bmatrix}
I_{H_1} & T\\
0 & I_{H_2}
\end{bmatrix}.
$$
Since the operator $I_{H_1}\oplus V_2$ is an isometry, we get $\|\Fog Uw\|=\|Gw\|$ for every $w\in H_1\oplus H_2$. As $U$ maps unit vectors of $H_1\oplus H_2$ onto unit vectors of $N(\Fog)^\perp$, we get
$$
\gamma(\Fog)=\inf_{\|w\|=1}\|Gw\|=m(G).
$$
Hence $\Fog$ attains $\gamma(\Fog)$ if and only if $G$ attains $m(G)$. Now $G$ is invertible and $G^{-1}=
\begin{bmatrix}
I_{H_1} & -T \\
0 & I_{H_2}
\end{bmatrix}$. Thus $G$ attains its minimum modulus if and only if $G^{-1}$ attains its norm.

Let $D=\operatorname{diag}(I_{H_1},-I_{H_2})$, which is unitary. Then $G^{-1}=DGD$, and hence $G^{-1}$ attains its norm if and only if $G$ attains its norm. Therefore, $G$ attains its minimum modulus if and only if $G$ attains its norm.  

\smallskip
By Theorem \ref{NA Foguel}, it follows that $G$ attains its norm if and only if $T$ is norm attaining. This completes the proof.
\end{proof}

In the following example, we observe that $\Fog$ may not be reduced minimum attaining even if $T$ attains its minimum or reduced minimum.

\begin{example}
    Let $H=\ell^2(\N)$ and $T=\operatorname{diag}(\lambda_n)$ with $\lambda_n=2-\frac{1}{n}$ for $n\geq 1$.
Then $T$ is injective with
$m(T)=\gamma(T)=1=\|Te_1\|$. Thus $T$ attains both its minimum modulus and its
reduced minimum modulus. But $\|T\|=2$ and $\|Tx\|<2\|x\|$ for all nonzero $x\in H$, and hence $T\notin \na(H)$.
By Theorem~\ref{thm:RMA}, $\Fog=\begin{bmatrix}
    V^* & T \\
    0 & V
\end{bmatrix}$ does not attain $\gamma(\Fog)$
for any isometry $V$. This example shows that $\Fog$ may not attain its reduced minimum modulus even if $T$ is minimum or reduced minimum attaining.
\end{example}

Now we provide a complete characterization of $2\times 2$ positive block operator matrix to be $\am$.

\begin{theorem}\label{AM block iff}
       Let $T=
\begin{bmatrix}
A_1 & X\\
X^* & A_2
\end{bmatrix}
\in \clb(H_1\oplus H_2)_+.
$
Then the following are equivalent:

\begin{enumerate}
\item $T\in \am(H_1\oplus H_2)$;

\item $A_j\in \am(H_j)_+$ with
$\sigma_{\rm ess}(A_1)=\sigma_{\rm ess}(A_2)=\{\alpha\}$
for some $\alpha\ge0$, and
$$X=F- (A_1^{\alpha-}+F_1)^{\frac{1}{2}}\,C\,(A_2^{\alpha-}+F_2)^{\frac{1}{2}},$$ where $C$ is a contraction, $F\in\clf(H_2,H_1)$ and $F_j\in\clf(H_j)_{sa}$ such that $A_j^{\alpha-}+F_j\in \clk(H_j)_+$.
\end{enumerate}
\end{theorem}

\begin{proof}
(i)$\implies$(ii): Let $T\in \am(H_1\oplus H_2)$. Then by \cite[Theorem 5.8]{AM Cha}, we have $T=\alpha I_{H_1 \oplus H_2} -K +F$ for some $\alpha \geq 0$, $K\in \clk(H_1 \oplus H_2)_{+}$ and $F\in \clf(H_1 \oplus H_2)_{+}$. Let $K=\begin{bmatrix}
        K_{11} & K_{12} \\
        K_{21} & K_{22}
    \end{bmatrix}$ and $F=\begin{bmatrix}
        F_{11} & F_{12} \\
        F_{21} & F_{22}
    \end{bmatrix}$. Since $K,F \geq 0$, we have $K_{jj},F_{jj}\geq 0$ for $j=1,2 ;$ $K_{21}=K^*_{12}$ and $F_{21}=F^*_{12}$. Also, there exists a contraction $C$ such that $K_{12}=K_{11}^\frac{1}{2}CK_{22}^\frac{1}{2}$.    Then
    \begin{equation*}
        \begin{split}
            T&=\begin{bmatrix}
                \alpha I_{H_1} - K_{11} + F_{11} & -K_{12}+F_{12}\\
                -K^*_{12}+F^*_{12} & \alpha I_{H_2} - K_{22}+F_{22}
            \end{bmatrix}.
        \end{split}
    \end{equation*}
    Thus  $A_j =\alpha I_{H_j} - K_{jj} + F_{jj}$ for $j=1,2$ and $X=F_{12}-K_{12}$. By \cite[Theorem 6.2]{AN Closure}, we have 
    $$A_j-\alpha I_{H_j}=(A_j-\alpha I_{H_j})^+ - (A_j-\alpha I_{H_j})^-,$$ where $A_j^{\alpha+}=(A_j-\alpha I_{H_j})^+\in\clf(H_j)_+$ and $A_j^{\alpha-}=(A_j-\alpha I_{H_j})^-\in \clk(H_j)_+$. Therefore, we get 
    $$
    F_{jj}-K_{jj}= (A_j-\alpha I_{H_j})^+ - (A_j-\alpha I_{H_j})^-.
    $$
    Write $K_{jj}=(A_j-\alpha I_{H_j})^- + F_j \in \clk(H_j)_+$, where $F_j=F_{jj}- (A_j-\alpha I_{H_j})^+\in\clf(H_j)_{sa}$. Consequently, $K_{12}=(A_1^{\alpha-}+F_1)^{\frac{1}{2}}\,C\,(A_2^{\alpha-}+F_2)^{\frac{1}{2}}$. Writing $F:=F_{12}\in\clf(H_2,H_1)$, we obtain $$X=F-(A_1^{\alpha-}+F_1)^{\frac{1}{2}}\,C\,(A_2^{\alpha-}+F_2)^{\frac{1}{2}}.$$

    \medskip
    (ii)$\implies$(i): Assume that (ii) holds. Since $A_j\in\am(H_j)_{+}$ with $\es(A_j)=\{\alpha\}$ for $j=1,2$, by \cite[Theorem 6.2]{AN Closure}, it follows that $$A_j=\alpha I_{H_j}+A_j^{\alpha+}-A_j^{\alpha-},$$
    where $A_j^{\alpha+}\in \clf(H_j)_{+}$ and $A_j^{\alpha-}\in \clk(H_j)_{+}$ for $j=1,2$. Also, given that $$X=F-(A_1^{\alpha-}+F_1)^{\frac{1}{2}}\,C\,(A_2^{\alpha-}+F_2)^{\frac{1}{2}},$$ where $C$ is a contraction, $F\in\clf(H_2,H_1)$ and $F_j\in\clf(H_j)_{sa}$ such that $\widetilde{K}_j:=A_j^{\alpha-}+F_j\in \clk(H_j)_{+}$.
\smallskip

    Now observe that $$\widetilde{K}:=\begin{bmatrix}
        \widetilde{K}_1 & \widetilde{K}^\frac{1}{2}_1 C \widetilde{K}^\frac{1}{2}_2\\
        \widetilde{K}^\frac{1}{2}_2 C^* \widetilde{K}^\frac{1}{2}_1 & \widetilde{K}_2
    \end{bmatrix}\in \clk(H_1\oplus H_2)_{+}$$ and $$\widetilde{F}:=\begin{bmatrix}
        A_1^{\alpha+}+F_1 & F\\
        F^* & A_2^{\alpha+}+F_2
    \end{bmatrix}\in \clf(H_1\oplus H_2)_{sa}.$$
 Then we can write $T$ by the following way 
 \begin{align*}
     T&=\begin{bmatrix}
         A_1 & X\\
         X^* & A_2
     \end{bmatrix}\\
     &=\begin{bmatrix}
         \alpha I_{H_1}+A_1^{\alpha+}-A_1^{\alpha-} & F-(A_1^{\alpha-}+F_1)^{\frac{1}{2}}\,C\,(A_2^{\alpha-}+F_2)^{\frac{1}{2}}\\
         F^*-(A_2^{\alpha-}+F_2)^{\frac{1}{2}}\,C^*\,(A_1^{\alpha-}+F_1)^{\frac{1}{2}} & \alpha I_{H_2}+A_2^{\alpha+}-A_2^{\alpha-}
     \end{bmatrix}\\
     &=\begin{bmatrix}
             \alpha I_{H_1} & 0 \\
             0 & \alpha I_{H_2}
         \end{bmatrix} + \begin{bmatrix}
        A_1^{\alpha+}+F_1 & F\\
        F^* & A_2^{\alpha+}+F_2
    \end{bmatrix}\\
    &-\begin{bmatrix}
        A_1^{\alpha-}+F_1 & (A_1^{\alpha-}+F_1)^{\frac{1}{2}}\,C\,(A_2^{\alpha-}+F_2)^{\frac{1}{2}}\\
    (A_2^{\alpha-}+F_2)^{\frac{1}{2}}\,C^*\,(A_1^{\alpha-}+F_1)^{\frac{1}{2}} & A_2^{\alpha-}+F_2
    \end{bmatrix}\\
    &= \alpha I_{H_1\oplus H_2} + \widetilde{F} - \widetilde{K}.
 \end{align*}
Thus by \cite[Proposition 3.5]{AN Toeplitz}, it follows that $T\in \am(H_1\oplus H_2).$
\end{proof}

Next, we provide the characterization of Foguel-type operators to be $\am$.

\begin{theorem}\label{AM Foguel}
Let $V_j\in\clb(H_j)$ be an isometry for $j=1,2$, and $T\in\clb(H_2,H_1)$.
Consider the Foguel-type operator $\Fog=\begin{bmatrix}
V_1^* & T\\
0 & V_2
\end{bmatrix}$ on $H_1\oplus H_2$.
Then the following are equivalent:
\begin{enumerate}
    \item $\Fog\in\am(H_1\oplus H_2)$;
    
    \item $\dim N(V_1^*)<\infty$ and $T$ is of finite rank.
\end{enumerate}
\end{theorem}

\begin{proof}
Since $V_2$ is an isometry, we have $$\Fog^*\Fog=\begin{bmatrix}
V_1V_1^* & V_1T\\
T^*V_1^* & T^*T+I_{H_2}
\end{bmatrix}.$$
Let $A_1=V_1V_1^*, A_2=T^*T+I_{H_2}$ and $X=V_1T$. Note that $V_1V_1^*$ is an orthogonal projection onto $R(V_1)$. Consequently, $V_1V_1^*=I-P_{N(V_1^*)}$, where $P_{N(V_1^*)}$ denotes the orthogonal projection onto
$N(V_1^*)$. Hence we can write
$$
\Fog^*\Fog=I_{H_1\oplus H_2}+\begin{bmatrix}
-P_{N(V_1^*)} & V_1T\\
T^*V_1^* & T^*T
\end{bmatrix}.
$$

\medskip
(i)$\implies$(ii): Assume that $\Fog$ is $\am$. Then by \cite[Theorem 5.14]{AM Cha}, $\Fog^*\Fog\in\am(H_1\oplus H_2)_+$. By the similar argument as in the $\an$-case and using Theorem \ref{AM block iff}, we can show that $\es(A_1)=\es(A_2)=\{1\}$, and  $N(V_1^*)$ is finite dimensional.

Again, by Theorem \ref{AM block iff}, we have $X=F-(A_1^{\alpha-}+F_1)^{\frac{1}{2}}\,C\,(A_2^{\alpha-}+F_2)^{\frac{1}{2}}$ for some contraction $C$, $F\in\clf(H_2,H_1)$ and $F_j\in\clf(H
_j)_{sa}$ such that $A_j^{\alpha-}+F_j\in \clk(H_j)_+$. 

Here $\alpha=1$ and $A_2^{\alpha-}=(A_2-\alpha I_{H_2})^{-}=(T^*T+I_{H_2}-I_{H_2})^{-}=0$. Consequently, $X=V_1T$ becomes a finite rank operator, and hence $T=V_1^*(V_1T)$ has finite rank.

\medskip
(ii)$\implies$(i): Assume that (ii) holds. Then by the similar argument as in $\an$ case and using \cite[Lemma 3.8]{Carvajal2}, \cite[Theorem 5.14]{AM Cha}, we obtain $\Fog\in\am(H_1\oplus H_2)$.
\end{proof}

\begin{example}
    Let $H=\ell^2(\N)$ with standard orthonormal basis $\{e_n\}_{n\geq 1}$, and $S$ is the unilateral shift defined by $Se_n=e_{n+1}$ for $n\in\N$. Then $S$ is an isometry with $N(S^*)=\C e_1$. We write $e_i\otimes e_j$ for the rank one operator $x\mapsto\langle x,e_j\rangle e_i$ on $H$.

    Let $T=e_1\otimes e_1$ and $\Fog=\begin{bmatrix}S^* & e_1\otimes e_1\\
    0 & S
    \end{bmatrix}$. Since $\dim N(S^*)=1$ and $T$ has rank one, so by Theorem \ref{AN Foguel} and Theorem \ref{AM Foguel}, it follows that $\Fog$ is $\an$ as well as $\am$.

    Here $\|T\|=1$, and consequently $\|\Fog\|=\frac{1+\sqrt{5}}{2}$, the golden ratio. Moreover, Proposition \ref{reduced minmod} gives $\gamma(\Fog)=\frac{\sqrt{5}-1}{2}$ and it is attained, by Theorem \ref{thm:RMA}.
\end{example}

\section{Fredholmness, Moore–Penrose inverse and Applications}
In this section, we study the Fredholm properties of Foguel-type operators and determine their Moore--Penrose inverses.
We characterize Fredholmness in terms of the dimensions of
the defect spaces of the two isometries and compute the
corresponding Fredholm index. We also obtain an explicit
formula for the Moore--Penrose inverse and establish a
connection between its norm attainment property, reduced
minimum attainment of the original operator, and norm
attainment of $T$. We conclude by applying our results to Toeplitz and Hankel operators on the Hardy space.
\begin{proposition}\label{fredholm foguel}
Let $V_j\in \clb(H_j)$ be an isometry for $j=1,2$, and $T\in \clb(H_2,H_1)$. Consider the
Foguel-type operator $\Fog=
\begin{bmatrix}
V_1^* & T\\
0 & V_2
\end{bmatrix}$ on $H_1\oplus H_2$.
Then $\Fog$ is Fredholm if and only if $\dim N(V_j^*)<\infty$. Furthermore, $\mathrm{ind}(\Fog)=\dim N(V_1^*)-\dim N(V_2^*)$.
\end{proposition}

\begin{proof}
We already have $N(\Fog)=N(V_1^*)\oplus\{0\}$.
Thus $\dim N(\Fog)=\dim N(V_1^*)$. Let $(x,y)\in N(\Fog^*)$. Since $\Fog^*
=\begin{bmatrix}
V_1 & 0\\
T^* & V_2^*
\end{bmatrix}$, we have
$$
V_1x=0 \,\text{ and }\, T^*x+V_2^*y=0.
$$
Thus $x=0$ and $y\in N(V_2^*)$, and so $N(\Fog^*)=\{0\}\oplus N(V_2^*)$.
Hence $\dim N(\Fog^*)=\dim N(V_2^*)$.

Furthermore, as $\Fog$ always has closed range, it follows that
$$
\Fog\text{ is Fredholm}
\iff
\dim N(V_j^*)<\infty.
$$
Consequently, we obtain
$$
\mathrm{ind}(\Fog)
=\dim N(V_1^*)-\dim N(V_2^*).
$$
\end{proof}

\begin{example}\label{ex fred}
Let $H=\ell^2(\mathbb N)$ with orthonormal basis $\{e_n\}_{n\ge1}$. Let $S$ be the
unilateral shift on $H$. Consider the operator $\Fog=\begin{bmatrix}
V_1^{*} & T\\ 
0 & V_2
\end{bmatrix}$ on $H\oplus H$, where $V_1=S^2$, $V_2=S$ and $T=\mathrm{diag}\left(\frac{1}{n}\right)$ for $n\in\N$. Then $N(V_1^{*})=\mathrm{span}\{e_1,e_2\}$ and $N(V_2^*)=\mathrm{span}\{e_1\}$.
By Proposition \ref{fredholm foguel}, $F$ is Fredholm with index one.

Since $\|T\|=1=\|Te_1\|$, we get $T\in\na(H)$.
Hence by Theorem \ref{NA Foguel} and Theorem \ref{thm:RMA}, $\Fog$ attains its norm as well as its reduced minimum modulus.
But $T$ does not have finite rank, so by Theorem \ref{AN Foguel} and Theorem \ref{AM Foguel},
$\Fog\notin\an(H\oplus H)$ and $\Fog\notin\am(H\oplus H)$, even though $\dim N(V_1^*)<\infty$.
Thus $\Fog$ is Fredholm, norm attaining and reduced minimum attaining, but neither $\mathcal{AN}$ nor $\mathcal{AM}$. Since $T$ is compact, $\Fog\in\overline{\an(H\oplus H)}$, by Theorem \ref{closure-AN-Foguel}.
\end{example}

\begin{example}\label{ex:nonfred}
Let $H_1=\ell^2(\mathbb Z)$ with orthonormal basis $\{e_n\}_{n\in\mathbb Z}$ and let
$W$ be the bilateral shift on $H_1$, defined by $We_n=e_{n+1}$ for $n\in\mathbb Z$. Let $H_2=\ell^2(\N)$ with
orthonormal basis $\{e_n\}_{n\ge1}$. Consider the operators $V_1=W$, $V_2e_n=e_{2n}$ for $n\in\N$ and $Tx=\langle x,e_1\rangle\,e_0$ for $x\in H_2$.
Then $V_2$ is an isometry on $H_2$. Now it is easy to observe that
$$
V_2^*e_m=\begin{cases} e_{m/2}, & m \text{ even},\\ 0, & m \text{ odd},\end{cases}
$$
and $N(V_2^*)=\overline{\mathrm{span}}\{e_{2k-1}:k\in\N\}$,
which is infinite dimensional. The operator $T\in \clb(H_2,H_1)$ has rank one and $\|T\|=1$.
Let $\Fog=\begin{bmatrix}
V_1^* & T\\
0 & V_2
\end{bmatrix}$ on $H_1\oplus H_2$. Since $W$ is unitary, $N(W^*)=\{0\}$, and so
$N(\Fog)=N(W^*)\oplus\{0\}=\{0\}\oplus \{0\}$, whereas
$N(\Fog^*)=\{0\}\oplus N(V_2^*)$ is infinite dimensional. By Proposition \ref{fredholm foguel}, $\Fog$ is not Fredholm.
 
We have $\dim N(V_1^*)=0<\infty$ and $T$ has finite rank. Consequently, by Theorem \ref{AN Foguel} and Theorem \ref{AM Foguel}, it follows that $\Fog$ is $\an$ as well as $\am$.

As $V_1$ is unitary, $\Fog$ is injective. Thus $m(\Fog)=\gamma(\Fog)=\frac{\sqrt5-1}{2}>0$ by Proposition \ref{reduced minmod} and Remark \ref{closed range and unitary V_1}.

This example shows that $\an$ and $\am$ conditions depend only on $\dim N(V_1^*)$ and $T$, while the Fredholm
property also depends on $\dim N(V_2^*)$.
\end{example}

Next, we give the Moore-Penrose inverse for Foguel-type operators.

\begin{theorem}\label{pseudo}
Let $V_j\in \clb(H_j)$ be an isometry for $j=1,2$, and $T\in \clb(H_2,H_1)$. Consider the Foguel-type operator
$\Fog=\begin{bmatrix}
V_1^* & T\\
0 & V_2
\end{bmatrix}$ on $H_1\oplus H_2$. Then the Moore-Penrose inverse of $\Fog$ is given by
\begin{equation}
\Fog^\dagger=\begin{bmatrix}
V_1 & -V_1TV_2^*\\
0 & V_2^*
\end{bmatrix}.
\end{equation}
\end{theorem}

\begin{proof}
    We have already observed that $N(\Fog)=N(V_1^*)\oplus \{0\}$ and $N(\Fog)^\perp=R(V_1)\oplus H_2$. From Proposition \ref{fredholm foguel}, we have $N(\Fog^*)=\{0\}\oplus N(V_2^*)$. As $R(\Fog)$ is always closed, we get $R(\Fog)=N(\Fog^*)^\perp=H_1\oplus R(V_2)$. For any $(x,y)\in H_1\oplus H_2$, write $$(x,y)=\underbrace{(x,\,V_2V_2^*y)}_{\in\,H_1\oplus R(V_2)=R(F)}+\underbrace{(0,\,(I-V_2V_2^*)y)}_{\in\,\{0\}\oplus N(V_2^*)=R(F)^\perp}.$$
    If $P_{R(\Fog)}$ denotes the orthogonal projection onto $R(\Fog)$, then we get $P_{R(\Fog)}(x,y)=(x,V_2V_2^*y)$. Let $\Fog^\dagger(x,y)=(z,w)$. Then $(z,w)\in N(\Fog)^\perp=R(V_1)\oplus H_2$ and $\Fog(z,w)=P_{R(\Fog)}(x,y)$. Consequently, it follows that $(V_1^*z+Tw, V_2w)=(x,V_2V_2^*y)$. Therefore, $$V_1^*z+Tw=x\, \text{ and } \,V_2w=V_2V_2^*y.$$
    From the second equation we get $w=V_2^*y$. Substituting this in the first equation we get $V_1^*z+TV_2^*y=x$. Since $z\in R(V_1)$, write $z=V_1z^{'}$ for some $z^{'}\in H_1$. Thus $V_1^*V_1z^{'}+TV_2^*y=x$, which further implies $z^{'}=x-TV_2^{*}y$, that is, $z=V_1x-V_1TV_2^*y$. Hence we obtain $\Fog^\dagger(x,y)=(V_1x-V_1TV_2^*y, V_2^*y)$. Writing in the block operator matrix, it follows that $\Fog^\dagger=\begin{bmatrix}
        V_1 & -V_1TV_2^* \\
        0 & V_2^*
    \end{bmatrix}$.
\end{proof}

\begin{proposition}
    Let $V_j\in \clb(H_j)$ be an isometry for $j=1,2$, and $T\in\clb(H_2,H_1)$. Consider the Foguel-type operator $\Fog=\begin{bmatrix}
        V_1^{*} & T \\
        0 & V_2
    \end{bmatrix}$. Then the following are equivalent:
    \begin{enumerate}
        \item $\Fog^\dagger$ is norm attaining;

        \item $\Fog$ is reduced minimum attaining;

        \item $T$ is norm attaining.
    \end{enumerate} 
\end{proposition}

\begin{proof}
    (i)$\iff$(ii) follows from \cite[Theorem 3.2]{SHK-GR2020}. On the other hand, by Theorem \ref{thm:RMA}, we obtain (ii)$\iff$(iii).
\end{proof}

\begin{corollary}\label{prop:unitaryinvertible}
Let $V_j\in \clb(H_j)$ be unitary for $j=1,2$, and $T\in \clb(H_2,H_1)$. Then $\Fog=\begin{bmatrix}
    V_1^{*} & T \\
    0 & V_2
\end{bmatrix}$ is invertible with
$$
\Fog^{-1} = \begin{bmatrix} V_1 & -V_1TV_2^*\\
0 & V_2^*
\end{bmatrix}.
$$
\end{corollary}

Since $\dim N(V_j^*)=0<\infty$ automatically when $V_j$ is unitary, Theorem \ref{AN Foguel} and Theorem \ref{AM Foguel} specialize as follows.

\begin{corollary}\label{cor:unitaryANAM}
Let $V_j\in \clb(H_j)$ be unitary for $j=1,2$, and $T\in \clb(H_2,H_1)$. Then $\Fog\in
\an(H_1\oplus H_2)$ if and only if $\Fog\in\am(H_1\oplus H_2)$ if and only
if $T$ is of finite rank.
\end{corollary}

\subsection{Applications to Hardy space operators}
Let $\mathbb{D}=\{z\in\mathbb{C}:|z|<1\}$ be the open unit disc and
$\mathbb{T}=\{z\in\mathbb{C}:|z|=1\}$ the unit circle. Let $L^{2}(\mathbb{T})$ denote the space of
square integrable functions on $\mathbb{T}$ with respect to the normalized Lebesgue measure $\mu$, and
$L^{\infty}(\mathbb{T})$ the Banach space of essentially bounded measurable functions on $\mathbb{T}$.
For $f\in L^{2}(\mathbb{T})$, let $\widehat f(n):=\langle f,\chi_n\rangle$ be its $n$-th Fourier coefficient,
where $\chi_n(z)=z^{n}$ for $z\in\mathbb{T}$ and $n\in\mathbb{Z}$.

The \emph{Hardy space} $H^{2}$ is the space of analytic functions
$f(z)=\sum_{n\geq 0}a_nz^{n}$ on $\mathbb{D}$ such that $\sum_{n\geq 0}|a_n|^{2}<\infty$, with the norm
$\|f\|_{2}=(\sum_{n\geq 0}|a_n|^{2})^{1/2}$. Every $f\in H^{2}$ has a radial limit
$\widetilde{f}(\zeta)=\lim_{r\to1^{-}}f(r\zeta)$ for almost every $\zeta\in\mathbb{T}$, and $f\mapsto \widetilde{f}$
identifies $H^{2}$ isometrically with the closed subspace
$$
\{g\in L^{2}(\mathbb{T}):\ \widehat g(n)=0\ \text{ for all } n<0\}
$$
of $L^{2}(\mathbb{T})$, the coefficients $a_n$ corresponding to $\widehat{\widetilde{f}}(n)$. We use this identification
throughout and write $f$ for $\widetilde{f}$ on $\T$. The functions $\{z^{n}\}_{n\ge0}$ form an orthonormal basis
of $H^{2}$, and for $w\in\mathbb{D}$, the point evaluation is given by $f(w)=\langle f,k_w\rangle$ with
$k_w(z)=(1-\overline{w}z)^{-1}$.

The Hardy space $H^{\infty}$ consists of the bounded analytic functions on $\mathbb{D}$. In terms of boundary functions, we write
$H^{\infty}=H^{2}\cap L^{\infty}(\mathbb{T})$.

A function $\theta\in H^{\infty}$ is called \emph{inner} if $|\theta(\zeta)|=1$ for almost every
$\zeta\in\mathbb{T}$. For $a\in\mathbb{D}\setminus\{0\}$
the function $b_a(z)=\frac{|a|}{a}\,\frac{a-z}{1-\overline{a}z}$ is analytic on a neighbourhood of
$\overline{\mathbb{D}}$ with $|b_a|=1$ on $\mathbb{T}$, and we write $b_0(z)=z$. A \emph{finite Blaschke product}
is a function $B=c\,b_{a_1}\cdots b_{a_m}$ with $|c|=1$ and $a_1,\dots,a_m\in\mathbb{D}$ not necessarily
distinct. It is an inner function, and its zeros in $\mathbb{D}$ are $a_1,\dots,a_m$. Also $\deg B =m$ is the number of its zeros in $\mathbb{D}$ counted with multiplicity. Every inner function that is not a finite Blaschke
product either has infinitely many zeros in $\mathbb{D}$ or has a nontrivial singular inner factor; see
\cite{Modelspace2016} for further details.

For $\vp\in L^\infty(\T)$, the Toeplitz operator $T_\vp$ and the Hankel operator $H_\vp$ on $H^2$ are respectively defined by
$$
T_\vp f=P(\vp f), \qquad f\in H^2,
$$
and
$$
H_\vp f=J(I-P)(\vp f), \qquad f\in H^2,
$$
where $P:L^2(\T)\to L^2(\T)$ denotes the orthogonal projection onto $H^2$,
and $J:L^2(\T)\to L^2(\T)$ is the unitary operator defined by
$$
J(z^n)=z^{-n-1},\qquad n\in\mathbb{Z}.
$$
We refer to \cite{Douglas,Rosenthal} for further details.
\begin{proposition}\label{Toeplitz case}
    Let $\vp\in L^\infty(\T)$ and $\,\Fog_\vp:=\begin{bmatrix}
        T_z^* & T_\vp \\
        0 & T_z
    \end{bmatrix}$ be a Foguel-type operator on $H^2\oplus H^2$. Then 
    \begin{enumerate}
        \item $\Fog_\vp$ is norm attaining if and only if $\vp=\|\vp\|_{\infty}\overline{\Theta}_1\Theta_2$, where $\Theta_1$ and $\Theta_2$ are inner functions with no common inner factor.
        \item $\Fog_\vp$ is $\an$ if and only if $\vp=0$ a.e. on $\T$.
     \end{enumerate}
\end{proposition}

\begin{proof}
\textbf{Proof of (i):} Since the Toeplitz operator $T_z$ is an isometry, by Theorem \ref{NA Foguel}, we have
$$
\Fog_\vp \in \na(H^2\oplus H^2)
\iff T_\vp\in\na(H^2).
$$
Equivalently, by
\cite[Theorem 4]{Yoshino2002}, we can say $\varphi=\|\varphi\|_{\infty}\,\overline{\Theta}_1\Theta_2$,
where $\Theta_1$ and $\Theta_2$ are inner functions having no common inner factor.

\smallskip
\textbf{Proof of (ii):} Note that $T_z$ is an isometry with $N(T^*_z)=\C1$. Thus by Theorem \ref{AN Foguel}, we get $\Fog_\vp$ is $\an$ if and only if $T_\vp$ has finite rank, that is, $\vp=0$ a.e. on $\T$.
\end{proof}

\begin{proposition}\label{hankel}
    Let $\vp\in L^\infty(\T)$ and $\,\Fog_\vp=\begin{bmatrix}
        T_z^* & H_\vp \\
        0 & T_z
    \end{bmatrix}$ be a Foguel-type operator on $H^2\oplus H^2$. Then 
    \begin{enumerate}
        \item $\Fog_\vp$ is norm attaining if and only if $\vp=dist(\vp,H^\infty)(\Theta + \eta)$, where $\Theta\in H^\infty$ and $\eta\in L^\infty(\T)$ with $|\eta|=1$ a.e. on $\T$ and $0\in \sigma_p(T_\eta)$.
        
        \item $\Fog_\vp$ is $\an$ if and only if $(I-P)\vp$ is a rational function.

        \item $\Fog_\vp\in
        \overline{\an(H^2\oplus H^2)}$ if and only if $\vp\in H^\infty+C(\T)$.
     \end{enumerate}
\end{proposition}

\begin{proof}
\textbf{Proof of (i):} Using \cite[Theorem 5]{Yoshino2002} the proof of (i) is similar to Proposition \ref{Toeplitz case}.

\smallskip
\textbf{Proof of (ii):} Since $T_z$ is isometry and $\dim N(T^*_z)=1$, by Theorem \ref{AN Foguel}, $\Fog_\vp$ is $\an$ if and only if $H_\vp$ has finite rank. By \cite[Corollary 3.2, Page 21]{Peller}, it is equivalent that $(I-P)\vp$ is a rational function.  

\smallskip
\textbf{Proof of (iii):} By Theorem \ref{closure-AN-Foguel}, $\Fog_\vp\in \overline{\an(H^2\oplus H^2)}$ if and only if $H_\vp$ is compact. Equivalently, by \cite[Theorem 5.5, Page 27]{Peller}, we obtain $\vp\in H^\infty + C(\T)$.
\end{proof}

\begin{example}
Consider $\vp(z)=\sum_{n\geq 1}2^{-n}\,\bar z^{\,2^n}$ for $z\in\T$.
The series converges absolutely and uniformly on $\T$, so $\vp$
is continuous. In particular $\vp\in H^\infty+C(\T)$, and by
Proposition \ref{hankel}(iii), we get $\Fog_\vp\in\overline{\an(H^2\oplus H^2)}$.

Since $\vp$ has only negative Fourier coefficients, $(I-P)\vp=\vp$. As $\vp$ is not a rational function, by
Proposition \ref{hankel}(ii), it follows that $\Fog_\vp\notin\an(H^2\oplus H^2)$.

On the other hand, for $\vp(z)=\bar{z}^2$ on $\T$, we have $H_\vp=z\otimes 1 + 1\otimes z$, which has rank two. Consequently, $\Fog_\vp$ is $\an$.
\end{example}

\begin{theorem}\label{thm:inner-foguel}
Let $\theta_1,\theta_2$ be nonconstant inner functions, and $T\in\mathcal{B}(H^{2})$. Consider the Foguel-type operator $\Fog_{\theta_1,\theta_2;T}=\begin{bmatrix} T_{\theta_1}^{*} & T\\ 0 & T_{\theta_2}\end{bmatrix}$ on $H^2\oplus H^2$.
Write $\clk_{\theta_j}:=H^{2}\ominus\theta_j H^{2}$ for the model space associated with $\theta_j$ for $j=1,2$. Then the following hold:
\begin{enumerate}
\item $N(\Fog_{\theta_1,\theta_2;T})=\clk_{\theta_1}\oplus\{0\}$, $\;N(\Fog_{\theta_1,\theta_2;T}^{*})=\{0\}\oplus\, \clk_{\theta_2}$, and
$R(\Fog_{\theta_1,\theta_2;T})=H^{2}\oplus\theta_2H^{2}$ is closed.
\item $\Fog_{\theta_1,\theta_2;T}$ is Fredholm if and only if $\theta_1$ and $\theta_2$ are finite Blaschke products, and then
$$
\mathrm{ind}\,(\Fog_{\theta_1,\theta_2;T})=\deg\theta_1-\deg\theta_2 .
$$
\item $\Fog_{\theta_1,\theta_2;T}\in\mathcal{AN}(H^{2}\oplus H^{2})$ if and only if
$\Fog_{\theta_1,\theta_2;T}\in\mathcal{AM}(H^{2}\oplus H^{2})$ if and only if $\theta_1$ is a finite Blaschke product and $T$ has finite rank.
In particular, these properties do not depend on $\theta_2$.
\end{enumerate}
\end{theorem}

\begin{proof}
Since $\theta_j$ is inner, $T_{\theta_j}$ is the isometry with $R(T_{\theta_j})=\theta_jH^{2}$ and
$N(T_{\theta_j}^{*})=H^{2}\ominus\theta_jH^{2}=\clk_{\theta_j}$ for $j=1,2$.

\textbf{Proof of (i):} This is the computation of $N(\Fog)$, $N(\Fog^{*})$ and $R(\Fog)$ in Proposition \ref{fredholm foguel} and Theorem \ref{pseudo}
with $V_1=T_{\theta_1}$ and $V_2=T_{\theta_2}$.

\textbf{Proof of (ii):} By Proposition~\ref{fredholm foguel}, $\Fog_{\theta_1,\theta_2;T}$ is Fredholm if and only if
$\dim \clk_{\theta_j}<\infty$ for $j=1,2$. Now from \cite[Proposition~5.19]{Modelspace2016}, we have $\dim \clk_{\theta_j}<\infty$ if and only if $\theta_j$ is a finite Blaschke product for $j=1,2$. In that case, by
\cite[Proposition~5.16]{Modelspace2016}, $\dim \clk_{\theta_j}$ equals the number of zeros of $\theta_j$ in $\D$,
counted with multiplicity, that is, $\dim \clk_{\theta_j}=\deg\theta_j$. Hence $\Fog_{\theta_1,\theta_2;T}$ is
Fredholm if and only if $\theta_1$ and $\theta_2$ are finite Blaschke products.

By part (i), it follows that
$\dim N(\Fog_{\theta_1,\theta_2;T})=\dim \clk_{\theta_1}=\deg\theta_1$ and
$\dim N(\Fog_{\theta_1,\theta_2;T}^{*})=\dim \clk_{\theta_2}=\deg\theta_2$. Consequently, we obtain
$$
\mathrm{ind}\,(\Fog_{\theta_1,\theta_2;T})=\deg\theta_1-\deg\theta_2 .
$$

\textbf{Proof of (iii):} By Theorem \ref{AN Foguel} and Theorem \ref{AM Foguel}, each property is equivalent to
$\dim N(T_{\theta_1}^{*})<\infty$ and $T$ of finite rank. Since $N(T_{\theta_1}^{*})=\clk_{\theta_1}$, the first condition says $\theta_1$ is a finite Blaschke product by the argument in (ii).
\end{proof}

\begin{example}
Consider the function $u(z)=\exp\,\bigl(\tfrac{z+1}{z-1}\bigr)$. Then $u$ is a singular inner function. Since $u$ has no zeros in
$\mathbb{D}$ and is nonconstant, it is not a finite Blaschke product. Hence  by
\cite[Proposition 5.19]{Modelspace2016}, $\dim \clk_u=\infty$.
Let $\theta_1=z$, $\theta_2=u$ and $T=1\otimes1$. By Theorem \ref{thm:inner-foguel}(iii),
$\Fog_{z,u;T}$ is $\an$ as well as $\am$. But $N(\Fog_{z,u;T}^{*})=\{0\}\oplus \clk_u$ is infinite
dimensional. Consequently, $\Fog_{z,u;T}$ is not Fredholm.

If we choose $\theta_1=u$, $\theta_2=z$ and the same $T$ above. Then $\Fog_{u,z;T}$ is neither $\an$ nor $\am$. 
\end{example}
\section*{Declarations}

\subsection*{Conflict of Interest}
The authors declare that there is no conflict of interest relevant to this article.

\subsection*{Funding}
Not Applicable.

\subsection*{Author Contributions} 
Both authors contributed equally to the conceptualization and preparation of the manuscript.

\subsection*{Data Availability}
Not Applicable.

\section*{Acknowledgements}
The authors used ChatGPT (OpenAI) to assist with improving the language and polishing the grammar. All mathematical ideas, results, proofs, and conclusions are the authors' own, and the authors take full responsibility for the contents of the manuscript.


\begin{thebibliography}{99}

\bibitem{Apostol} C. Apostol, The reduced minimum modulus, Michigan Math. J. {\bf 32} (1985), no.~3, 279--294.

\bibitem{NR2021}
N. Bala and G. Ramesh, A representation of hyponormal absolutely norm attaining operators, Bull. Sci. Math. {\bf 171} (2021), Paper No. 103020, 15 pp.

\bibitem{NR} N. Bala and G. Ramesh, Hyperinvariant subspaces for normaloid essential isometric operators, J. Math. Anal. Appl. {\bf 543} (2025), no.~2, Paper No. 128998, 10 pp.

\bibitem{Bhuia} S.~R. Bhuia, Minimum and Essential Minimum Moduli of Generalized Quadratic Operators, Results Math. {\bf 81} (2026), no.~8, Paper No. 202.

\bibitem{Bouldin} R. Bouldin, The essential minimum modulus, Indiana Univ. Math. J. {\bf 30} (1981), no.~4, 513--517.

\bibitem{Carvajal1} X. Carvajal and W. Neves, Operators that achieve the norm, Integral Equations Operator Theory {\bf 72} (2012), no.~2, 179--195.

\bibitem{Carvajal2} X. Carvajal and W. Neves, Operators that attain their minima, Bull. Braz. Math. Soc. (N.S.) {\bf 45} (2014), no.~2, 293--312.

\bibitem{Douglas} R.~G. Douglas, {\it Banach algebra techniques in operator theory}, second edition, Graduate Texts in Mathematics, 179, Springer, New York, 1998.

\bibitem{Feldman} I.~A. Feldman, N.~Y. Krupnik and A.~S. Markus, On the norm of polynomials of two adjoint projections, Integral Equations Operator Theory {\bf 14} (1991), no.~1, 69--91.

\bibitem{Foguel} S.~R. Foguel, A counterexample to a problem of Sz.-Nagy, Proc. Amer. Math. Soc. {\bf 15} (1964), 788--790.

\bibitem{AM Cha} J. Ganesh, G. Ramesh and D. Sukumar, A characterization of absolutely minimum attaining operators, J. Math. Anal. Appl. {\bf 468} (2018), no.~1, 567--583.

\bibitem{Garcia2009} S.~R. Garcia, The norm and modulus of a Foguel operator, Indiana Univ. Math. J. {\bf 58} (2009), no.~5, 2305--2316.

\bibitem{Modelspace2016} S.~R. Garcia, J. Mashreghi and W.~T. Ross, {\it Introduction to model spaces and their operators}, Cambridge Studies in Advanced Mathematics, 148, Cambridge Univ. Press, Cambridge, 2016.

\bibitem{GindlerTaylor} H.~A. Gindler and A.~E. Taylor, The minimum modulus of a linear operator and its use in spectral theory, Studia Math. {\bf 22} (1962), 15--41.

\bibitem{GGK1} I. Gohberg, S. Goldberg and M.~A. Kaashoek, {\it Classes of linear operators. Vol. I}, Operator Theory: Advances and Applications, 49, Birkh\"auser, Basel, 1990.

\bibitem{Groetsch} C.~W. Groetsch, {\it Generalized inverses of linear operators: representation and approximation}, Monographs and Textbooks in Pure and Applied Mathematics, No. 37, Dekker, New York-Basel, 1977.

\bibitem{Halmos} P.~R. Halmos, On Foguel's answer to Nagy's question, Proc. Amer. Math. Soc. {\bf 15} (1964), 791--793.

\bibitem{SHK-GR2020} S.~H. Kulkarni and G. Ramesh, Operators that attain reduced minimum, Indian J. Pure Appl. Math. {\bf 51} (2020), no.~4, 1615--1631.
\bibitem{GR2018}
G. Ramesh, Absolutely norm attaining paranormal operators, J. Math. Anal. Appl. 465 (1) (2018) 547–556.  
\bibitem{Rosenthal} R.~A. Martinez~Avendano and P.~M. Rosenthal, {\it An introduction to operators on the Hardy-Hilbert space}, Graduate Texts in Mathematics, 237, Springer, New York, 2007.

\bibitem{Osaka} M.~S. Moslehian and H. Osaka, {\it Advanced techniques with block matrices of operators}, Frontiers in Mathematics, Birkh\"auser/Springer, Cham, [2024].

\bibitem{AN:Matrix2026} P. Nag and G. Ramesh, On absolutely norm (minimum) attaining $2\times 2$ block operator matrix, J. Math. Anal. Appl. {\bf 563} (2026), no.~1, Paper No. 130742, 20 pp.

\bibitem{Paulsen} S.~K. Pandey and V.~I. Paulsen, A spectral characterization of $\mathcal{AN}$ operators, J. Aust. Math. Soc. {\bf 102} (2017), no.~3, 369--391.

\bibitem{Pedersen} G.~K. Pedersen, {\it Analysis now}, Graduate Texts in Mathematics, 118, Springer, New York, 1989.

\bibitem{Peller82} V.~V. Peller, Estimates of functions of power bounded operators on Hilbert spaces, J. Operator Theory {\bf 7} (1982), no.~2, 341--372.

\bibitem{Peller} V.~V. Peller, {\it Hankel operators and their applications}, Springer Monographs in Mathematics, Springer, New York, 2003.

\bibitem{Pisier} G. Pisier, A polynomially bounded operator on Hilbert space which is not similar to a contraction, J. Amer. Math. Soc. {\bf 10} (1997), no.~2, 351--369.

\bibitem{Raghupathi2012} M. Raghupathi, A note on the norm and spectrum of a Foguel operator, arXiv:1206.6725 (2012).

\bibitem{AN Toeplitz} G. Ramesh and S.~S. Sequeira, Absolutely norm attaining Toeplitz and absolutely minimum attaining Hankel operators, J. Math. Anal. Appl. {\bf 516} (2022), no.~1, Paper No. 126497, 12 pp.

\bibitem{AN Closure} G. Ramesh and S.~S. Sequeira, On the closure of absolutely norm attaining operators, Linear Multilinear Algebra {\bf 71} (2023), no.~18, 2894--2914.

\bibitem{VR} D. Venku~Naidu and G. Ramesh, On absolutely norm attaining operators, Proc. Indian Acad. Sci. Math. Sci. {\bf 129} (2019), no.~4, Paper No. 54, 17 pp.

\bibitem{Yoshino2002} T. Yoshino, The structure of norm-achieved Toeplitz and Hankel operators, Nihonkai Math. J. {\bf 13} (2002), 43--55.

\end{thebibliography}
\end{document}